\documentclass[11pt,oneside]{amsart}
\usepackage[utf8]{inputenc}
\usepackage{enumitem}
\usepackage{csquotes}
\usepackage{bbm}
\usepackage{subcaption}
\usepackage{float}
\usepackage{fancyhdr}
\usepackage{etoolbox}
\usepackage{hyperref}
\usepackage[hyperpageref]{}%{backref}
\usepackage{amsmath, amssymb}
\usepackage{mathrsfs}
\usepackage{palatino} %DA TOGLIERE PER EJP
\usepackage{comment}
\usepackage[normalem]{ulem}

\numberwithin{equation}{section}  %numberwithin goes before cleverefs when using hyperref

\usepackage[top=2.5cm, bottom=2.5cm,left=2.5cm,right=2.5cm]{geometry} %DA TOGLIERE PER EJP

\usepackage{amsthm}

\theoremstyle{plain}
\newtheorem{theorem}{Theorem}
\newtheorem{corollary}[theorem]{Corollary}
\newtheorem{lemma}[theorem]{Lemma}
\newtheorem{proposition}[theorem]{Proposition}

\theoremstyle{remark}
\newtheorem{remark}{Remark}

\usepackage[svgnames]{xcolor}

\definecolor{bordeaux}{HTML}{cc0000}

\definecolor{viola}{rgb}{0.6, 0.4, 0.8}
\usepackage{pgfplots}

\def \1{\mathbbm{1}} % funziona solo con il package bbm.
 
 \def \R {{\mathbb R}}
 
 \def \N {{\mathbb N}}

 \def \E {{\mathbb E}}
 
 \def \cE {\mathcal{E}}

 \def \s {{\sigma}}
 
 \def \z {{\z}}
 \def \m {{\mu}}

 \def \z {{\zeta}}

 \def \G {{\Gamma}}

\def \bm {{\bf m}}
\usepackage{todonotes}

\pgfplotsset{compat=1.18} 

\title[Ising on 2-community SBM]{The Ising Model on a Two-Community Stochastic Block Model
}

\author{Alessandra Bianchi$^{*}$}
\author{Vanessa Jacquier$^{*}$}
\author{Matteo Sfragara$^{*}$}

\begin{document}

\numberwithin{equation}{section}
	\begin{abstract}
We study the Ising model on a two-community stochastic block model, where $n$ spins are split into two equal groups with inter-community interaction parameter $\alpha_n\in[0,1]$. We provide a complete characterization of the phase diagram and show that, almost surely with respect to the graph realization, the model undergoes a uniqueness/non-uniqueness phase transition of the Gibbs measure. In particular, in the supercritical regime, the law of the magnetization vector of the two communities converges to a mixture of Dirac measures that, depending on whether $\alpha_n\gg 1/n$ or $\alpha_n\lesssim1/n$, is supported on two or four points, with possibly different weights. In the uniqueness region, we further analyze the fluctuations of the magnetization vector in the subcritical regime and we prove a quenched central limit theorem.


    \par\smallskip\noindent
    {\it Acknowledgments.} 
    The authors thank Giacomo Passuello for useful discussions and suggestions. 
    The work of A.~Bianchi and V.~Jacquier is partially funded by the GNAMPA project ``Ising model on random graphs with community structure: equilibrium and dynamical transitions''.
    The work of A.~Bianchi is partially supported by the Italian Ministry of Foreign Affairs and International Cooperation, grant number BR26GR05. 	
	\end{abstract}

\maketitle
\begingroup
\renewcommand{\thefootnote}{\fnsymbol{footnote}}
\footnotetext[1]{\textit{Dipartimento di Matematica ``Tullio Levi-Civita'', 
Università di Padova, Via Trieste 63, 35121 Padova, Italy.}
\texttt{alessandra.bianchi@unipd.it, vanessa.jacquier@unipd.it, matteo.sfragara@unipd.it }}
\endgroup

\section{Introduction}\label{sec:intro}

Interacting spin systems constitute a fundamental class of models in statistical mechanics and probability theory, describing how local interactions at the microscopic level generate collective macroscopic behavior in large systems \cite{FV18, G81}. Among these, mean-field models have been extensively investigated due to their analytical tractability and their ability to capture key features such as the structure of equilibrium states, phase transitions, and fluctuation phenomena. A well-known example is the Curie–Weiss (CW) model, which has become a central model in the rigorous study of equilibrium statistical mechanics. Beyond their original physical motivation, mean-field spin systems have proved to be remarkably flexible modeling tools. Over the past decades, they have been successfully employed in diverse contexts, including social dynamics, opinion formation, decision-making processes, finance, ecology, and chemical systems. This broad range of applications highlights the versatility of these models: by appropriately modifying the structure of interactions or introducing disorder, one can analyze both physical, social, and biological phenomena in a mathematically rigorous setting.

A natural extension of the CW model arises when one considers several interacting populations, or communities, so that the system is partitioned into distinct groups with potentially different interaction strengths within and across groups. From the perspective of statistical mechanics, this generalization offers a richer landscape for equilibrium analysis: the phase diagram can exhibit multiple phase transitions, metastable states, and complex patterns of magnetization \cite{C14, GBC09, KT20}.

This viewpoint connects closely to random graph theory, which provides a powerful framework for modeling large systems in which collective behavior is shaped not only by interaction strengths but also by the topology of the underlying network. Within this setting, a particularly relevant class of random graphs is the Stochastic Block Model (SBM), in which vertices are divided into communities, producing stronger interactions within groups and weaker interactions between them. The SBM has become a central model in contexts where modular organization is essential \cite{A18, GN02, SBKNM19}. In network science, it serves as a reference model for community detection and for analyzing the spread of information, diseases, and behaviors across structured populations. In the life sciences, it is used to model metabolic and protein interaction networks, capturing how aggregation and modularity influence global behavior. More broadly, the SBM provides a theoretical
framework for investigating emergent phenomena driven by community structure.

In this paper, we study the Ising model on a two-community SBM (2-SBM) and investigate how the presence of interacting communities affects spontaneous magnetization and asymptotic fluctuations. We consider a mean-field spin system consisting of $n$ spins partitioned into two equally sized communities, where the interaction network is random and reflects the underlying community structure. 

Our first set of results concerns the thermodynamic limit and the phase diagram of the model. 
We characterize the limiting free energy and determine the asymptotic law of the magnetization vector as the system size tends to infinity. 
In particular, almost surely with respect to the graph realization, we show that a phase transition emerges as a function of the inverse temperature $\beta$ and of the inter-community interaction parameter $\alpha_n$. 
Explicitly, while at high temperature regimes the system exhibits uniqueness of the asymptotic Gibbs measure,  
in the low-temperature regime, the measure loses uniqueness and concentrates on multiple Dirac masses: four in the case $\alpha_n \lesssim 1/n$, with weights possibly of  different depth, and only two if $\alpha_n \gg 1/n$. 

As a refinement of this analysis, we further investigate the fluctuations of the magnetization vector in the subcritical/uniqueness regime. At high temperatures, we establish a quenched central limit theorem (CLT), showing that, under the diffusive scale $\sqrt{n}$, the magnetization converges to a Gaussian distribution whose covariance matrix captures both intra- and inter-community correlations. The quenched nature of the result emphasizes the robustness of this asymptotic behavior under disorder.

These findings extend the classical results for the one-community CW and Erd\"{o}s--R\'enyi CW models, both of which are well understood in terms of equilibrium properties \cite{BG93} and fluctuations \cite{KLS19, KLS20, KLS21, KLS22}.
Comparable results have also been obtained in the multi-community framework, but only under deterministic settings and assuming inter-community parameters that remain constant (i.e., independent of the system size) \cite{C14, LS18, KT20, KLSS20}, apart from the very recent work \cite{LL26}, which also considers size-dependent and random inter-community interactions.
In this paper, we address this gap by first analyzing the deterministic two-community setting with the inter-community parameter 
$\alpha_n$, and then incorporating randomness through the two-community stochastic block model (2-SBM).\\

\noindent
{\bf Outline of the paper.} In Section~2, we introduce the Ising model on the 2-SBM, 
present the main results, and outline the key ideas of their proofs. 
In Section~3, we analyze the bipartite CW model, establishing the corresponding phase diagram and 
proving convergence results for its magnetization.
We then return to the Ising model on the 2-SBM. In Section~4, we show that the random Gibbs measure concentrates 
around the corresponding bipartite CW Gibbs measure, and we establish the results on the asymptotic behavior of the magnetization. 
Finally, Section~5 is devoted to the proof of the fluctuation results for the magnetization at high temperatures.

%%%%%%%%%%%%%%%%%%%%%%%%%%%%%%%%%%%%%%%%%%%%%%%%%%%%%%%%%

\section{Model and results}\label{sec:ER}

\subsection{Mathematical model}
For any $n \in 2\mathbb{N}$, consider a two-community Stochastic Block Model (2-SBM) on $n$ vertices with two communities of  equal size $n/2$.  This is obtained by including every edge independently and according to Bernoulli laws with a parameter depending on whether the vertices belong to the same or different communities. Formally,  let $\mathscr{G}=(V,E)$ be a complete graph on $n$ vertices, with $V=[n]=\{1,\ldots,n\}$, and consider a partition of  $\mathscr{G}$ into two equal-sized communities $\mathscr{G}_1=(V_1, E_1)$ and $\mathscr{G}_2=(V_2, E_2)$, where $\mathscr{G}_1$ and $\mathscr{G}_2$  are complete graphs  on $n/2$ vertices. Without loss of generality we assume $V_1=\{1,\ldots, n/2\}$ and $V_2=\{n/2+1,\ldots, n\}$. Let $\cE_I$ and $\cE_E$ denote respectively the set of \textit{internal edges} and \textit{external edges}, that are identified by ordered pairs of vertices as
\begin{align*}
    &\cE_I=\{(i,j):\, i,j\in V_1 \, \text{  or  }\,  i,j\in V_2,\, i<j\} \qquad \text{and} \qquad \cE_E=\{(i,j):\, i\in V_1,\, j\in V_2 \} .
\end{align*}
We then consider two sequences of parameters $(p_n)_{n\in\N}\in(0,1]$ and $(\alpha_n)_{n\in\N}\in[0,1]$ such that:
\begin{itemize}
\item[-\,] Every edge in $\cE_I$ is included in the graph with  probability $p_n$;
\item[-\,] Every edge in $\cE_E$ is included in the graph with  probability $\alpha_n p_n$.
\end{itemize}
This procedure produces a 2-SBM, whose interaction matrix has intra- and inter-connection probabilities among vertices, given by
$p_n$  and $\alpha_n p_n$ respectively.\\

We now define the Ising model on the above 2-SBM. To this end, it is convenient to map the random procedure 
described above with two independent sequences  of Bernoulli i.i.d. random variables, 
$(\xi_{ij})_{(i,j)\in \cE_I}$ and $(\zeta_{ij})_{(i,j)\in \cE_E}$, 
with parameters $p_n$ and $p_n\alpha_n$, respectively. Let $\mathcal{X}_{n}=\{-1,+1\}^{n}$ be the configuration space of a $n$-spin system, 
with elements denoted by $\sigma=(\s_i)_{i\in V} \in \mathcal{X}_{n}$, where $\sigma_i\in\{-1,+1\}$ 
is the spin value at $i \in V$. 
The Hamiltonian of the Ising model on the 2-SBM is defined as
\begin{align}\label{def:hamil-DCW}
    H_{n}(\sigma)&=-\frac{1}{np_n} \left\{ \sum_{(i,j) \in \cE_I} \xi_{ij} \sigma_i \sigma_j + \sum_{(i,j) \in \cE_E} \zeta_{ij} \sigma_i \sigma_j \right \},
\end{align}
and the corresponding Gibbs probability measure is given by
\begin{align*}
    \mu_{n,\beta}(\sigma)=\frac{e^{-\beta H_n(\sigma)}}{Z_{n,\beta}}\quad \mbox{with } \quad Z_{n,\beta}=\sum_{\sigma\in\mathcal{X}_n} e^{-\beta H_n(\sigma)},
    \end{align*}
where $\beta> 0$ is the inverse temperature and $Z_{n,\beta}$ is referred to as the partition function.

Notice that all these quantities are random as they depend on the realization of the Bernoulli random variables $\xi_{ij}$'s and $\zeta_{ij}$'s, and thus intrinsically also on the value of $p_n$ and $\alpha_n$.
Throughout the paper, we will denote by $\Omega$ the probability space where these sequences are defined, and we write $\mathbb{P}$ (resp. $\E$) to denote their joint probability law (resp. expectation). We will usually hide the dependence from $\omega\in\Omega$,  although the presence of a source of randomness should be kept in mind. In this framework, the Gibbs measure $\mu_{n,\beta}$ is also referred to as a \textit{quenched Gibbs measure}, in contrast to the \textit{annealed Gibbs measure}, which is obtained by averaging $\mu_{n,\beta}$ with respect to $\mathbb{P}$.\\

We assume that
\begin{equation}\label{assump}
\lim_{n\to\infty}n p_n =+\infty \quad \mbox{and } \quad \lim_{n\to\infty}\alpha_n= \hat\alpha\in [0,1],
\end{equation}
and for notational convenience, we will simply write $p\equiv p_n$, while keeping the symbol $\alpha_n$
to avoid being confused with its limit $\hat\alpha$, and because we are interested in understanding the asymptotic behavior of the model depending on the different choice of $(\alpha_n)_{n\in\N}$.\\

An important immediate observation is that taking the $\mathbb{P}$-average of the Hamiltonian \eqref{def:hamil-DCW}, one obtains the following Hamiltonian of a bipartite Curie-Weiss (CW) model
\begin{align}\label{def:hamiltonian_CW}
    \bar{H}_{n}(\sigma)=-\frac{1}{n}\sum_{(i, j) \in\cE_I} \sigma_i \sigma_j -\frac{\alpha_n}{n} \sum_{(i, j) \in\cE_E}\sigma_i \sigma_j,
\end{align}
with corresponding Gibbs measure $ \bar{\mu}_{n,\beta}(\sigma)=\frac{e^{-\beta \bar{H}_n(\sigma)}}{\bar{Z}_{n,\beta}}$ and partition function $\bar{Z}_{n,\beta}$.
In this sense, the Ising model on the 2-SBM can be considered
as a randomized version of the bipartite CW model. This relationship will be indeed a key tool in the proofs, which basically rest on the concentration properties of the graph.\\

Before stating the main results and describing the phase diagram of the model, it is convenient to introduce some further notation. Let us consider the (random) \textit{finite- and infinite-volume free energies} of the model
\begin{align}\label{free_energy}
    f_{n,\beta}=\frac{1}{ n \beta} \log Z_{n,\beta} \quad \mbox{and } \quad     f_{\beta}=\lim_{n\to\infty}f_{n,\beta} ,
\end{align}
and define analogously the finite- and infinite-volume free energies for the bipartite CW model
\begin{align}\label{free_energy_CW}
    \bar{f}_{n,\beta}=\frac{1}{n \beta} \log \bar{Z}_{n,\beta},\quad \mbox{ and } \quad
    \bar{f}_{\beta}=\lim_{n \to \infty} \bar{f}_{n,\beta}.
\end{align}
Let us introduce the \textit{local magnetization}
of the community $\mathscr{G}_i$, for $i=1,2$,
\begin{align*}
    m_i^{(n)} (\sigma)= \frac{1}{n/2} \sum_{j \in V_i} \sigma_j,
\end{align*}
taking values in $\Gamma_{n/2}=\left \{ -1, -1+\frac{4}{n},...,1-\frac{4}{n}, 1 \right \}$, 
and let $\bm^{(n)}(\s)=(m_1^{(n)}(\s),m_2^{(n)}(\s))$ be the global \textit{magnetization vector} of $\s$,
taking values in $\Gamma_{n/2}^2$. 

Denoting by $\|\cdot\|$ the Euclidean norm in $\R^2$, for $\alpha \in [0,1]$ and $\bm=(m_1,m_2) \in [-1,1]^2$, define
\begin{align}\label{def:E}
    E_{\alpha}(\bm)=\frac{\|\bm\|^2}{4} +\frac{\alpha}{2} m_1m_2,
\end{align}
and note that $\bar{H}_n(\sigma)=-\frac{n}{2} E_{\alpha_n}(\bm^{(n)}(\s)) +1\,$.
However, the additive constant $+1$ in the  Hamiltonian does not affect the associated Gibbs measure, 
since it only produces a  constant multiplicative factor in the Boltzmann weight, which cancels out with the normalization. 
Therefore, in the following, we neglect this term and work with the Hamiltonian
\begin{align}\label{def:hamiltonian_CW_magnet}
\bar{H}_n(\sigma)=-\frac{n}{2} E_{\alpha_n}(\bm^{(n)}(\s)).
\end{align}
We will show that the infinite-volume free energy of the bipartite CW model is equal to
\begin{align*}
    \bar{f}_{\beta}=\frac{1}{2 \beta}\sup_{\bm \in [-1,1]^2} G_{\hat \alpha,\beta}(\bm),
\end{align*}
where, for $\alpha\in[0,1],\, \beta> 0$ and $\bm\in[-1,1]^2$, the free energy functional is given by
\begin{align}\label{def:free_energy_functional_CW}
    G_{\alpha,\beta}(\bm)= \beta E_{\alpha}(\bm )-\mathcal{I}(\bm),
\end{align}
and the entropic term $\mathcal{I}({\bf{m}})=\mathcal{I}(m_1)+\mathcal{I}(m_2)$
is defined by setting   
\begin{align*}
\mathcal{I}(m_i)=\frac{1+m_i}{2}\log \left ( \frac{1+m_i}{2} \right )+\frac{1-m_i}{2}\log \left ( \frac{1-m_i}{2} \right ),\quad \mbox{ for }\, i=1,2.
\end{align*}

\subsection{Results}
Our first result describes the convergence of the magnetization law and characterizes the phase diagram. Since the free energy functional $G_{\alpha,\beta}(\bm)$ may exhibit one, two, or four maximizers depending on the value of $\beta$ and $\alpha$, 
we introduce the following notation. For $m_g>0$ and $m_l>0$ defined,  respectively, 
in \eqref{eq:condition_mg} and \eqref{eq:condition_ml}  as implicit functions of $\alpha$ and $\beta$, we set
\begin{equation*}
\bm^0=(0,0), \quad \bm^1=-\bm^2 =(m_g,m_g), \quad \bm^3=-\bm^4=(m_l,- m_l)\,.
\end{equation*}
Here, $\bm^1, \bm^2$ are symmetric maximizers  corresponding to the \textit{global maximum} of $G_{\alpha,\beta}(\bm)$, 
while $\bm^3, \bm^4$ are anti-symmetric maximizers  corresponding to the \textit{local maximum}. 
We will show that they are all nonzero when $\beta$ is larger than some explicit critical value $\beta_c\equiv\beta_c(\alpha)$,  
and that if $\alpha=0$, then $m_g=m_l=m^*$,  where $m^*>0$ is the unique positive solution of the classical mean-field equation $\mathrm{atanh}(m)=\frac{\beta}{2}m$
(see Eq.\eqref{eq:expansion_mstar}).

Let $\mathcal{L}(\bm^{(n)})$ be the law of $\bm^{(n)}$ under the Gibbs measure $\mu_{n,\beta}$ and let $\delta_x$ be the Dirac measure at $x$.

\begin{theorem}[Convergence of the Magnetization Law]
\label{thm_convergence}
Under the assumptions in \eqref{assump}, the following statements hold almost surely with respect to $\mathbb{P}$:
\begin{itemize}
\item[(i)] For any $\beta > 0$, the infinite-volume free energy exists and coincides with the free energy $\bar{f}_\beta$ of the bipartite CW model defined in \eqref{free_energy_CW}, i.e.
\[
\lim_{n \to \infty} f_{n,\beta} = \bar{f}_\beta.
\]

\item[(ii)] A phase transition in the uniqueness of the infinite-volume Gibbs measure occurs at 
$\beta_c = \frac{2}{1+\hat{\alpha}}$.
More precisely:
\begin{itemize}
\item[(a)] If $\beta \le \beta_c$, then
$\lim_{n \to \infty} \mathcal{L}(\bm^{(n)}) = \delta_{\bm^0}$.

\item[(b)] If $\beta > \beta_c$, then
\[
\lim_{n \to \infty} \mathcal{L}(\bm^{(n)}) =
\begin{cases}
\frac{1}{2} \left( \delta_{\bm^1} + \delta_{\bm^2} \right), & \text{if } \alpha_n \gg \frac{1}{n}, \\[6pt]
\frac{1}{2(1+e^{-c})} \left( \delta_{\bm^1} + \delta_{\bm^2} \right)
+ \frac{e^{-c}}{2(1+e^{-c})} \left( \delta_{\bm^3} + \delta_{\bm^4} \right),
& \text{if } \alpha_n \lesssim \frac{1}{n},
\end{cases}
\, 
\]
where $c = \lim_{n \to \infty} n \alpha_n \ge 0$. \footnote{ Given two sequences $(a_n)_{n \in \mathbb{N}}$ and $(b_n)_{n \in \mathbb{N}}$,  we write $a_n \gg b_n$ if $\displaystyle{\lim_{n\to\infty} b_n/a_n=0}$, and $a_n \lesssim b_n$, if $\displaystyle{\lim_{n\to\infty}a_n/b_n=c< \infty}$.}
\end{itemize}
\end{itemize}
\end{theorem}
\begin{remark}[The role of $\alpha_n$]
The phase diagram described in Theorem~\ref{thm_convergence} is governed by the structure of the energy landscape, 
and in particular by the interplay between global and local maxima of the free energy functional. 
 The parameter $\alpha_n$ plays a crucial role in this picture, as it controls the interaction 
 between the two communities, and therefore the relative stability of these maxima. 
 To make this precise, we will show in Lemma \ref{lem:epsilonalpha} that, 
 as $n \to \infty$, the difference between the free energy functional at global and local maxima satisfies
 the following expansion in $\alpha_n\to 0$,
$$
\varepsilon(\alpha_n) = G_{\alpha_n,\beta}(\bm^1) - G_{\alpha_n,\beta}(\bm^3) = (m^*)^2\beta \, \alpha_n + o(\alpha_n)\,.
$$
In addition, in \eqref{eq:expansion_mstar} we obtain the expansions
\begin{align*}
m_g(\alpha_n)= m^* + \frac{\frac{\beta}{2}m^*}
{\frac{1}{1-(m^*)^2}-\frac{\beta}{2}}\alpha_n + o(\alpha_n), \qquad 
m_l(\alpha_n)= m^* -\frac{\frac{\beta}{2}m^*}
{\frac{1}{1-(m^*)^2}-\frac{\beta}{2}}\alpha_n + o(\alpha_n).
\end{align*}
As a consequence, when $\alpha_n \to 0$, both $m_g(\alpha_n)$ and $m_l(\alpha_n)$ converge to $m^*$, and asymptotically
they all correspond to the global maximum.
However, at finite $n$, the contribution to the Gibbs measure around each maximum is proportional to $e^{- n \varepsilon(\alpha_n)}$, 
so that the rate at which $\alpha_n$ vanishes determines the concentration properties of the measure. 
This mechanism explains the different regimes described in the theorem.
\end{remark}

\begin{remark}[Extremal cases]
In the extremal case $\alpha_n \equiv 0$, the model reduces to two independent Erd\"{o}s--R\'enyi CW communities on $n/2$ vertices, whose deterministic counterpart is described in \cite{C14}: for $\beta \le 2$ the unique maximizer is $(0,0)$, 
while for $\beta > 2$ four global maximizers (both symmetric and anti-symmetric) emerge.
At the opposite extreme, when $\alpha_n \equiv 1$, the model coincides with the Erd\"{o}s--R\'enyi CW model on $n$ vertices, and the classical picture is again recovered: for $\beta \le 1$ the unique maximizer is $(0,0)$, while for $\beta >1$ two symmetric global maximizers emerge (see \cite{BG93} for further details).\end{remark} 

\begin{remark}[Comparison with \cite{LL26}]
We note that the recent work \cite{LL26} analyzes a similar setting, 
deriving comparable results under the assumptions $\alpha_n \gg 1/n$ and $\alpha_n \ll 1/n$.
However, it does not consider the regime in which 
$\displaystyle{\lim_{n \to \infty}} n\alpha_n = c > 0$. 
This scaling is of particular interest, as the difference between global 
and local maxima vanishes as $\alpha_n \to 0$, while their contribution to the Gibbs weights remains non-negligible, 
being of order $n\alpha_n$. 
This behavior is reflected in the convergence in law of the magnetization, 
where different maximizers appear with non-trivial and distinct weights, as described in case~(ii-b).
\end{remark}

In the uniqueness region of the phase diagram, the above quenched distributional convergences 
can be straightened to the following quenched strong law of large numbers for $\bm^{(n)}$.  
\begin{corollary}[SLLN for $\bm^{(n)}$]
\label{cor}
Let $\beta\le \beta_c$. Under assumptions \eqref{assump}, almost surely with respect to $\mathbb{P}$,
\begin{equation*}
\bm^{(n)} \overset{a.s.}{\longrightarrow} \bm^0 \qquad \mbox{w.r.t. } \mu_{n,\beta}, \mbox{ as }n\to\infty.
\end{equation*}
\end{corollary}
The proof of this corollary relies on a classical argument of exponentially fast convergence and is postponed to Section \ref{subsec:convergence_mag}, together with the proof of Theorem \ref{thm_convergence}.
\\

Our second main result concerns the fluctuations of the magnetization in the high-temperature regime when $\beta < \beta_c$. 
\begin{theorem}[Quenched CLT for the magnetization]
\label{thm:CLT}
Let $\beta < \beta_c$. Under assumptions \eqref{assump}, 
$$
\mathcal{L}(\sqrt{n} \,\bm^{(n)}) \underset{n\to\infty}{\longrightarrow} \mathcal{N}(0, \Sigma) \qquad \text{in } \mathbb{P}-\text{ probability},
$$
where $\mathcal{N}(0, \Sigma)$ is a two-dimensional normal random vector with covariance matrix
\begin{equation}
\label{def:sigma}
\Sigma = \frac{2-\beta}{(1-\beta) + \frac{\beta^2(1-\hat\alpha^2)}{4}} 
\begin{pmatrix} 
1 & \frac{\beta\hat\alpha}{2-\beta} \\
\frac{\beta\hat\alpha}{2-\beta} & 1
\end{pmatrix}.
\end{equation}
\end{theorem}
\begin{remark}[Relation with previous results]
Theorem \ref{thm:CLT} generalizes to the two-community setting, the quenched central limit theorem for the magnetization of the Erd\"{o}s--R\'enyi CW model first proved in \cite{KLS19} and later improved in \cite{KLS20}.  
Note that the covariance matrix reduces to a diagonal matrix when $\hat\alpha = 0$, namely when the inter-community interaction vanishes,while when $\hat\alpha = 1$, it matches the variance term obtained in \cite{KLS20} for the Erd\"{o}s--R\'enyi CW model. 
We emphasize that similar results have recently been obtained in \cite{DM23} for a broad class of models that includes the present two-community setting. However, our results rely on weaker assumptions on the asymptotic behavior of $p_n$ as $n \to \infty$, as specified in \eqref{assump}.
\end{remark}

\begin{remark}[Fluctuations at the critical temperature]
As $\beta$ approaches the critical value $\beta_c = 2/(1+\hat{\alpha})$, the covariance matrix $\Sigma$ in \eqref{def:sigma} becomes singular, signaling a breakdown of the central limit theorem. Indeed, the system exhibits a non-standard behavior which becomes evident, with some further computations, when $\hat\alpha >0$. In that case, while the fluctuations of the difference $m_1 - m_2$ still scale as $n^{-1/2}$ and remain Gaussian, the fluctuations of the sum $m_1 + m_2$ are much larger, scaling as $n^{-1/4}$, and their limiting distribution is non-Gaussian. In analogy with the results for the Erd\"{o}s--R\'enyi CW model in \cite{KLS20}, we expect the limiting behavior to be determined by a delicate interplay between the sparsity of the graph, specifically the scaling of $p$, and the rate of convergence of the interaction parameter $\alpha_n \to \hat{\alpha}$. A rigorous investigation within this double-scaling window, which requires a much more involved analysis of the free energy landscape, goes beyond the scope of this paper and is left for future work.
\end{remark}

\subsection{Outlines of the proofs and main ingredients}
The proofs of Theorem \ref{thm_convergence} and Corollary \ref{cor} are based on the idea that the random Gibbs measure  
concentrates around the Gibbs measure of a bipartite CW model. 
Consequently, the first part of the analysis is devoted to the study of the Gibbs measure of the bipartite CW model and of its associated free energy functional, including the characterization of its global and local maximizers (see Section \ref{sec:CW}). 
In particular, the asymptotic behavior of the magnetization vector is obtained via a Taylor expansion of the free energy functional around its maximizers, combined with an analysis of the induced Gibbs measure on the magnetization space.
The second step consists in proving the aforementioned concentration (see Section \ref{sec:concentration}). 
This is achieved by comparing the original Hamiltonian $H_n$ with the averaged Hamiltonian $\bar H_n$, 
corresponding to the bipartite CW model. By controlling the difference $\Delta_n = H_n - \bar H_n$, 
one derives suitable exponential bounds, which imply that the Gibbs measures of the two systems are exponentially close. As a consequence, the free energy of the original system converges $\mathbb{P}$-almost surely to that of the bipartite CW model, 
and all limiting properties identified for the latter carry over to the original system.

Finally, to establish the results in Theorem \ref{thm:CLT}, we study the expectation of test functions of the rescaled magnetization under the Gibbs measure, expressed as a ratio of generalized and standard partition functions (see Section \ref{sec:fluctuations}). By introducing a suitably tilted partition function, this ratio can be decomposed into factors that can be analyzed separately: one factor captures the limiting behavior, while the others converge to one, allowing the identification of the limiting law.

%%%%%%%%%%%%%%%%%%%%%%%%%%%%%%%%%%%%%%%%%%%%%%%%%%%%%%%%%%%%%%%%%%%%%%%%%%%%%%%%%%%%%%%%%%%%%%%%%%%%%%%

\section{Bipartite Curie-Weiss model}
\label{sec:CW}

\subsection{Model and phase diagram}
The bipartite CW model is defined by the probability measure
\begin{align}\label{def:gibbs_CW}
    \bar{\mu}_{n,\beta}(\sigma)=\frac{e^{-\beta \bar{H}_n(\sigma)}}{\bar{Z}_{n,\beta}},
\end{align}
where $\bar{H}_n(\sigma)$ is defined in \eqref{def:hamiltonian_CW}.
Recall also the definition of free energy given in \eqref{free_energy_CW}. 

 The following main result echoes Theorem \ref{thm_convergence} and provides the asymptotic distribution of the magnetization, hence characterizing the phase diagram of the model. 
\begin{theorem}[Convergence of the magnetization law for CW]\label{thm_convergence_CW}
Let $\alpha_n>0$ and let $\lim_{n \to \infty} \alpha_n = \hat \alpha \in [0,1]$. For $\beta>0$, a phase transition in the uniqueness of the infinite-volume Gibbs measure occurs at 
$\beta_c = \frac{2}{1+\hat{\alpha}}$.
More precisely:
\begin{itemize}
\item[(a)] If $\beta \le \beta_c$, then
$\lim_{n \to \infty} \mathcal{L}(\bm^{(n)}) = \delta_{\bm^0}$.

\item[(b)] If $\beta > \beta_c$, then
\[
\lim_{n \to \infty} \mathcal{L}(\bm^{(n)}) =
\begin{cases}
\frac{1}{2} \left( \delta_{\bm^1} + \delta_{\bm^2} \right), & \text{if } \alpha_n \gg \frac{1}{n}, \\[6pt]
\frac{1}{2(1+e^{-c})} \left( \delta_{\bm^1} + \delta_{\bm^2} \right)
+ \frac{e^{-c}}{2(1+e^{-c})} \left( \delta_{\bm^3} + \delta_{\bm^4} \right),
& \text{if } \alpha_n \lesssim \frac{1}{n},
\end{cases}
\]
where $c = \lim_{n \to \infty} n \alpha_n \ge 0$.
\end{itemize}
\end{theorem}
It is worth emphasizing that closely related results were already established in \cite{C14} and subsequently extended in \cite{KT20, LS18} for a bipartite CW model with interaction parameters independent of $n$.
In contrast, the explicit dependence on $n$ considered here gives rise 
to novel phenomena that do not appear in \cite{C14, KT20, LS18}, and constitute 
the central focus of this work.

In the next section, we present a collection of preliminary results that play a key role in the proof of Theorem \ref{thm_convergence_CW}, while also being of independent interest.

\subsection{Preliminary results}
\subsubsection{The free energy}
We provide a variational characterization of the infinite-volume free energy. 
\begin{lemma}\label{lem:free_energy_CW} 
Let $\lim_{n\to\infty} \alpha_n = \hat \alpha \in [0,1]$ and $G_{\alpha,\beta}$ be the function defined in \eqref{def:free_energy_functional_CW}.
The infinite-volume free energy is equal to
\begin{align*}
    \bar{f}_{\beta}=\frac{1}{2 \beta}\sup_{\bm \in [-1,1]^2} G_{\hat\alpha,\beta}(\bm).
\end{align*}
\end{lemma}

\begin{proof}
From \eqref{def:hamiltonian_CW_magnet}, we have
\begin{align*}
    \bar{Z}_{n,\beta}&= \sum_{\sigma\in \mathcal{X}_n} e^{-\beta \bar{H}_n(\sigma)} =\sum_{\bm=(m_1,m_2)} \binom{n/2}{\frac{1+m_1}{2} n/2} \binom{n/2}{\frac{1+m_2}{2} n/2} e^{\beta\frac{n}{2} E_{\alpha_n}(\bm)}.
\end{align*}
Notice that the binomial coefficients satisfy the general bounds, for some constants $c,C>0$,
\begin{equation}\label{StirBound}
\frac{c}{\sqrt{n}}e^{- \frac{n}{2}\mathcal{I}(m_i)}\le\binom{n/2}{\frac{1+m_i}{2} n/2}\le Ce^{- \frac{n}{2}\mathcal{I}(m_i)},\qquad \mbox{ for } i=1,2.
\end{equation}
Since we are interested in the behavior of the partition function on the exponential scale and since it is a sum of only $|\Gamma_{n/2}^2 |= (n+1)^2$ positive terms, we can estimate $\bar{Z}_{n,\beta}$ by keeping only its
dominant term, and get from the above estimates that
\begin{align*}
  \frac{c}{\sqrt{n}}\exp\left \{\frac{n}{2} \sup_{\bm \in [-1,1]^2} G_{\alpha_n,\beta} (\bm) \right \} \leq \bar{Z}_{n,\beta} 
    & \leq C(n+1)^2 \exp\left \{\frac{n}{2} \sup_{\bm \in [-1,1]^2} G_{\alpha_n,\beta} (\bm) \right \},
\end{align*}
where the lower bound is obtained by the continuity in $\bm$ of $G_{\alpha_n,\beta}(\bm)$ through a standard argument.
By the definition of free energy in \eqref{free_energy_CW}, taking the limit for $n\to\infty$ 
and using that $\lim_{n \to \infty}G_{\alpha_n,\beta}=G_{\hat\alpha, \beta}$ uniformly in $n$,
we finally get
\begin{align*}
    \bar{f}_{\beta} &= \frac{1}{\beta}\lim_{n \to \infty} \frac{1}{n} \log \bar{Z}_{n,\beta} =
    \frac{1}{2\beta}\lim_{n\to\infty} \sup_{\bm \in [-1,1]^2} G_{\alpha_n,\beta} (\bm) 
    = \frac{1}{2\beta}\sup_{\bm \in [-1,1]^2} G_{\hat\alpha,\beta} (\bm).
\end{align*}
\end{proof}

\subsubsection{Maxima of the free energy functional}
In view of Lemma \ref{lem:free_energy_CW}, the analysis of the maximal values of  $G_{\alpha,\beta}(\bm)$ becomes central to providing a precise asymptotic formula for the partition function and then establishing proper limit theorems for magnetization.
\begin{proposition}\label{proposition:minima1}
Let $\beta>0$ and $\alpha \in (0,1]$. Define
\begin{align}\label{def:betastar}
    \beta^* (\alpha)= \frac{2}{1-\alpha} \frac{\operatorname{arctanh}(t_*)}{t_*},
\end{align}
where $t_* \in (0,1)$ is such that
\begin{align}\label{def:tstar} 
   (1-t_*^2)\frac{\operatorname{arctanh}(t_*)}{t_*}=\frac{1-\alpha}{1+\alpha}.
\end{align}
Then the functional $G_{\alpha,\beta}(\bm)$ defined in \eqref{def:free_energy_functional_CW} admits one, two, or four local maximizers. In particular:
    \begin{itemize}
        \item[(i)] for $\beta\le \frac{2}{1+\alpha}$, $\bm^0$ is the unique maximizer;
        \item[(ii)] for $\beta \in \left (\frac{2}{1+\alpha}, \beta^*(\alpha)\right ]$, there are two (symmetric) global maximizers, $\bm^1$ and $\bm^2$;
        \item[(iii)] for $\beta > \beta^*(\alpha)$, there are two (symmetric) global maximizers, $\bm^1$ and $\bm^2$, and two (anti-symmetric) local maximizers, $\bm^3$ and $\bm^4$.
    \end{itemize}
\end{proposition}
To prove the above proposition, we need the following technical lemma. 
\begin{lemma}\label{lem:elementary-ineq}
For $u \in (0,1)$, define
$\phi(u)=\frac{\operatorname{arctanh}(u)}{u}$ and 
$\psi(u)=(1-u^2)\frac{\operatorname{arctanh}(u)}{u}$.
Then:
\begin{enumerate}[label=\rm(\roman*)]
\item $\phi(u)<\dfrac{1}{1-u^2}$,  for every $u\in(0,1)$;
\item $\psi$ is strictly decreasing on $(0,1)$, with $\lim_{u \to 0}\psi(u)=1$ and 
%\qquad \text{ and } \qquad 
$\lim_{u \to 1}\psi(u)=0$.
\end{enumerate}
\end{lemma}
\begin{proof}
For part (i), let 
$f(u)=\frac{u}{1-u^2}-\operatorname{arctanh}(u)$,
and note that 
$f'(u)=\frac{2u^2}{(1-u^2)^2}>0$.
Hence $f(u)>0$ for every $u\in(0,1)$, and (i) follows. 

For part (ii), a direct computation gives
$\psi'(u)=\frac{u-(1+u^2)\operatorname{arctanh}(u)}{u^2}$.
Since $\operatorname{arctanh}(u)>u$ for every $u\in(0,1)$,
$
(1+u^2)\operatorname{arctanh}(u)>u,
$
hence $\psi'(u)<0$ for every $u\in(0,1)$. The endpoint limits follow from the standard expansion
$
\operatorname{arctanh}(u)=u+O(u^3)
$.
\end{proof}

\begin{proof}[Proof of Proposition \ref{proposition:minima1}]
We consider the stationary points of $G_{\alpha,\beta}(\bm)$ defined as
\begin{equation}\label{eq:grad}
\begin{split}
\frac{\partial G_{\alpha,\beta}}{\partial m_1}(\bm)
&= 
\beta\Big(\tfrac{1}{2}m_1 +\tfrac{\alpha}{2}m_2\Big)
- \operatorname{arctanh}(m_1)=0,
\\[2mm]
\frac{\partial G_{\alpha,\beta}}{\partial m_2}(\bm)
&= 
\beta\Big(\tfrac{1}{2}m_2 + \tfrac{\alpha}{2}m_1\Big)
- \operatorname{arctanh}(m_2)=0.
\end{split}
\end{equation}
Equivalently,
\begin{align}\label{eq:system_m}
m_1=\tanh\!\Big(\frac{\beta}{2}(m_1+\alpha m_2)\Big),
\qquad
m_2=\tanh\!\Big(\frac{\beta}{2}(\alpha m_1+m_2)\Big).
\end{align}
The Hessian matrix is
\begin{align*}
    D^2G_{\alpha,\beta}(\bm)=\begin{pmatrix}
        \frac{\beta}{2}-\frac{1}{1-m_1^2} & \frac{\beta \alpha}{2} \\
        \frac{\beta \alpha}{2} & \frac{\beta}{2}-\frac{1}{1-m_2^2}
    \end{pmatrix}\,.
\end{align*}

First we show that $\bm^0$ is the unique maximizer if and only if $\beta \leq \frac{2}{1+\alpha}$. By computing the Hessian in $\bm^0$, we obtain
\begin{equation}\label{eq:Hessian} D^2G_{\alpha,\beta}(0,0)=\begin{pmatrix}
        \frac{\beta}{2}-1& \frac{\beta \alpha}{2} \\
        \frac{\beta \alpha}{2} & \frac{\beta}{2}-1
    \end{pmatrix}\,,
\end{equation}
and the claim follows immediately from the signs of the two eigenvalues $\lambda_\pm=\frac{\beta}{2} (1\pm \alpha)-1$.

Next, we establish the existence of two symmetric critical points, $\bm^1$ and $\bm^2$, which are global maximizers whenever $\beta > \frac{2}{1+\alpha}$. Let $\bm^1=(m_g,m_g)$, with $m_g>0$, be the symmetric solution of $\nabla G_{\alpha,\beta}(\bm)=0$. Using \eqref{eq:system_m}, we derive the condition
\begin{align}\label{eq:condition_mg}
    m_g=\tanh\!\Big(\frac{\beta}{2}(1+\alpha)m_g\Big),
\end{align}
which has nonzero solution if and only if $\beta > \frac{2}{1+\alpha}$. When this holds, by computing the Hessian in $\bm^1$, we obtain
$$   D^2G_{\alpha,\beta}( m_g, m_g)=
\begin{pmatrix}
        \frac{\beta}{2}-\frac{1}{1-m_g^2} & \frac{\beta \alpha}{2} \\
        \frac{\beta \alpha}{2} & \frac{\beta}{2}-\frac{1}{1-m_g^2}
    \end{pmatrix}\,,
    $$
hence the eigenvalues are equal to $\lambda_{\pm}=\frac{\beta}{2} (1\pm \alpha)-\frac{1}{1-m_g^2}$.
By \eqref{eq:condition_mg},
$$
    \frac{\beta}{2}(1+\alpha)=\frac{\operatorname{arctanh}(m_g)}{m_g} \quad \Longrightarrow\quad
    \lambda_+=\frac{\operatorname{arctanh}(m_g)}{m_g}-\frac{1}{1-m_g^2}.
$$
By Lemma \ref{lem:elementary-ineq}(i), it follows that $\lambda_-<\lambda_+<0$. 
Since both eigenvalues are negative, $\bm^1$ is a maximizer. By symmetry, the same holds for $\bm^2=(-m_g,-m_g)$.

It remains to show that there exists a value $\beta^*(\alpha)$, defined as in \eqref{def:betastar}, such that for $\beta > \beta^*(\alpha)$, two additional (antisymmetric) maximizers, $\bm^3$ and $\bm^4$, appear. 
Let $\bm^3=(m_l,-m_l)$, with $m_l>0$, be the antisymmetric solution of $\nabla G_{\alpha,\beta}(\bm)=0$. Using \eqref{eq:system_m}, we derive the condition
\begin{align}\label{eq:condition_ml}
    m_l=\tanh\!\Big(\frac{\beta}{2}(1-\alpha )m_l\Big),
\end{align}
which has nonzero solution if and only if $\beta > \frac{2}{1-\alpha}$. When this holds, by computing the Hessian in $\bm^3$, we obtain
$$
D^2G_{\alpha,\beta}( m_l, -m_l)=\begin{pmatrix}
        \frac{\beta}{2}-\frac{1}{1-m_l^2} & \frac{\beta \alpha}{2} \\
        \frac{\beta \alpha}{2} & \frac{\beta}{2}-\frac{1}{1-m_l^2}
    \end{pmatrix}\,,
$$
which has  eigenvalues $\lambda_{\pm}=\frac{\beta}{2} (1\pm \alpha)-\frac{1}{1-m_l^2}$
By \eqref{eq:condition_ml} and Lemma \ref{lem:elementary-ineq}(i), we get
$$
    \frac{\beta}{2}(1-\alpha)=\frac{\operatorname{arctanh}(m_l)}{m_l}
    \quad \Longrightarrow\quad
    \lambda_-=\frac{\operatorname{arctanh}(m_l)}{m_l}-\frac{1}{1-m_l^2}<0\,.
$$
We finally show that $\lambda_+<0$ whenever $\beta>\beta^*(\alpha)$. 
By \eqref{eq:condition_ml},  we get
$$ 
\lambda_+=\frac{1+\alpha}{1-\alpha} \frac{\operatorname{arctanh}(m_l)}{m_l}-\frac{1}{1-m_l^2},
$$
hence $\lambda_+<0$ if and only if
$$
    (1-m_l^2)\frac{\operatorname{arctanh}(m_l)}{m_l}<\frac{1-\alpha}{1+\alpha}\,.
$$
Note that the left-hand side corresponds to the function $\psi$ defined in Lemma \ref{lem:elementary-ineq}, and the right-hand side is in $(0,1)$. By Lemma \ref{lem:elementary-ineq}(ii) there exists a unique $t_* \in (0,1)$ such that 
\begin{align*}
    \begin{cases}
        \lambda_+>0, &\qquad \text{ if } \,\,m_l<t_*, \\
        \lambda_+=0, &\qquad \text{ if } \,\,m_l=t_*, \\
        \lambda_+<0, &\qquad \text{ if } \,\,m_l>t_*.
    \end{cases}
\end{align*}
Since by \eqref{eq:condition_ml} the solution $t_*$ is strictly increasing in $\beta$, there exists a unique value $\beta^*(\alpha)$ such that
\begin{align}\label{beta*}
\beta^*(\alpha)=\frac{2}{(1-\alpha)}\frac{\operatorname{arctanh}(t_*)}{t_*}\,.
\end{align}
In conclusion, $\lambda_+<0$ for any $\beta<\beta^*(\alpha)$.
\end{proof}

The extremal case when $\beta > 0$ and $\alpha = 0$, already analyzed in \cite{LS18}, can be recovered from Lemma \ref{lem:critical0} by letting $\alpha_n \to 0$ as $n \to \infty$. Indeed, in this limit, the intermediate region $\left ( \frac{2}{1+\alpha_n}, \beta^*(\alpha_n)\right]$ collapses, since $\frac{2}{1+\alpha_n} \to 2$ and $\beta^*(\alpha_n) \to 2$, as we show in the next lemma. As a consequence, the system exhibits a single maximizer at high temperatures and four maximizers at low temperatures.
Moreover, the value of the antisymmetric local maxima becomes comparable to that of the global ones, 
up to a factor $\varepsilon(\alpha_n)$  which vanishes as $\alpha_n \to 0$ (see \eqref{def:M1M2} and Lemma \ref{lem:epsilonalpha}). 
In conclusion, in the low temperature regime, there are asymptotically four global maximizers.

\begin{lemma}\label{lem:critical0}
Let $(\alpha_n)_{n \in \N} \in (0,1]$ such that $\lim_{n \to \infty} \alpha_n=0$. Then $\beta^*(\alpha_n)$ defined in \eqref{def:betastar} satisfies
\[
\lim_{n \to \infty}\beta^*(\alpha_n)=2.
\]
\end{lemma}
\begin{proof}
    By Taylor expanding around $0$ both terms of the identity \eqref{def:tstar}, we readily obtain
$$  
1-\frac{2}{3}t^2+O(t^4)= 1-2\alpha_n+2\alpha_n^2+O(\alpha_n^3),
$$ 
which implies
$$
t_*^2=3\alpha_n(1-\alpha_n)+O(\alpha_n).
$$
Plugging this value in \eqref{beta*}, and applying again a Taylor expansion of $\operatorname{arctanh}(t_*)$ around $0$, we conclude that
    \begin{align*}
        \beta^*(\alpha_n)%=\frac{2}{1-\alpha_n} \left (1+\frac{t^2}{3}+O(t^4) \right)
        =\frac{2}{1-\alpha_n} \left (1+\alpha_n(1-\alpha_n)+O(\alpha_n)\right)
        \to 2\,,\quad \mbox{ as } \alpha_n \to 0\,.
    \end{align*}
\end{proof}

While $G_{\alpha,\beta}(\bm^0)=2\log2$ for any $\beta$ and $\alpha$, 
by explicit computation, let us introduce the following notation 
for the value of $G_{\alpha,\beta}$ in its local or global maximizers:
\begin{equation}\label{Max}
\begin{split}
&M_1(\alpha)=G_{\alpha,\beta}(\bm^1)\equiv G_{\alpha,\beta}(\bm^2),\\
&M_2(\alpha)=G_{\alpha,\beta}(\bm^3)\equiv G_{\alpha,\beta}(\bm^4).
\end{split}
\end{equation}
Then we define the energy gap
\begin{equation}
\label{def:M1M2}
 \varepsilon(\alpha)= M_1(\alpha)-M_2(\alpha).
\end{equation}
The following lemma shows that if $\lim_{n \to \infty}\alpha_n=0$, then $\varepsilon(\alpha_n)$  decreases linearly with $\alpha_n$. 
\begin{lemma}\label{lem:epsilonalpha}
Let $(\alpha_n)_{n\in\N}\in (0,1]$ such that $\lim_{n\to\infty}\alpha_n=0$, and for $\beta>\beta^*(\alpha_n)$, 
let $m^*>0$ be the unique positive solution of the fixed-point equation $\mathrm{atanh}(m)=\frac{\beta}{2}m$.
Then, for any $n\in \N$,
\begin{align*}
\varepsilon(\alpha_n)=(m^*)^2\beta \, \alpha_n + o(\alpha_n).
\end{align*}
\end{lemma}

\begin{proof}
The stationary points of the functional $G_{\alpha_n,\beta}(\bm)$ are computed in the proof of Proposition \ref{proposition:minima1}. 
In particular, they are obtained by solving the system $\nabla G_{\alpha_n,\beta}(\bm)=0$, i.e., \eqref{eq:grad}.
Equating them to zero yields the system
\begin{align}\label{eq:stationary1}
\operatorname{arctanh}(m_1) &= \frac{\beta}{2}\,(m_1 + \alpha_n m_2),\qquad
\operatorname{arctanh}(m_2) = \frac{\beta}{2}\,(m_2 + \alpha_n m_1).
\end{align}  
For the symmetric maximizers $\bm^{1}$ and $\bm^{2}$, which have $m_1=m_2=\pm m_g$, \eqref{eq:stationary1} gives that 
\begin{align*}
\Phi_{\mathrm{sym}}(m_g,\alpha_n) =\operatorname{arctanh}(m_g) - \frac{\beta}{2}(1+\alpha_n)m
_g= 0.
\end{align*}
For the antisymmetric maximizers $\bm^{3}$ and $\bm^{4}$, which have $m_1=-m_2=\pm m_l$, \eqref{eq:stationary1} gives
\begin{align*}
\Phi_{\mathrm{as}}(m_l,\alpha_n) =\operatorname{arctanh}(m_l) - \frac{\beta}{2}(1-\alpha_n)m_l =0.
\end{align*}
By explicit computations, we get
\begin{align*}
\partial_m \Phi_{\mathrm{sym}}(m^*,0)
&=
\partial_m \Phi_{\mathrm{as}}(m^*,0)
=
\frac{1}{1-(m^*)^2} - \frac{\beta}{2},
\\[2mm]
\partial_{\alpha_n} \Phi_{\mathrm{sym}}(m^*,0)
&= -\frac{\beta}{2}m^*,
\qquad
\partial_{\alpha_n}\Phi_{\mathrm{as}}(m^*,0)
= +\frac{\beta}{2}m^*.
\end{align*}
Let us explicitly stress the dependence of the maximizers on $\alpha_n$ by introducing the notation $m_g = m_g(\alpha_n)$ and $m_l = m_l(\alpha_n)$. Since $\partial_m\Phi_{\mathrm{sym}}(m^*,0)\neq 0$ for $\beta>2$, the Implicit Function Theorem implies that there exists $\tilde\alpha>0$ such that, for all $\alpha_n \in  [0,\tilde\alpha)$, the functions $m_g(\alpha_n), m_l(\alpha_n) \in C^1([0,\tilde\alpha))$. Differentiating \(\Phi_{\mathrm{sym}}(m_g(\alpha_n),\alpha_n)=0\) and \(\Phi_{\mathrm{as}}(m_l(\alpha_n),\alpha_n)=0\) with respect to $\alpha_n$, yields
\begin{align*}
\frac{d m_g}{d\alpha_n}\Big|_{\alpha_n=0}
&=
-\frac{\partial_{\alpha_n} \Phi_{\mathrm{sym}}(m^*,0)}
{\partial_m \Phi_{\mathrm{sym}}(m^*,0)}
=
\frac{\frac{\beta}{2}m^*}
{\frac{1}{1-(m^*)^2}-\frac{\beta}{2}},
\\[4mm]
\frac{d m_l}{d\alpha_n}\Big|_{\alpha_n=0}
&=
-\frac{\partial_{\alpha_n} \Phi_{\mathrm{as}}(m^*,0)}
{\partial_m \Phi_{\mathrm{as}}(m^*,0)}
=
-\frac{\frac{\beta}{2}m^*}
{\frac{1}{1-(m^*)^2}-\frac{\beta}{2}}\,.
\end{align*}
Hence the Taylor expansions near $0$ are
\begin{equation}\label{eq:expansion_mstar}
\begin{split}
m_g(\alpha_n)
&= m^* + s\alpha_n + o(\alpha_n),
\\[2mm]
m_l(\alpha_n)
&= m^* + a\alpha_n + o(\alpha_n),
\end{split}
\end{equation}
with $s=-a=\frac{\frac{\beta}{2}m^*}
{\frac{1}{1-(m^*)^2}-\frac{\beta}{2}}$.

\medskip

We now proceed to analyze the energy gap $\varepsilon(\alpha_n).$ We start by noticing first the identity
\begin{equation*}
G_{\alpha_n,\beta}(\bm)
=
G_{0,\beta}(\bm)
+\frac{\beta\alpha_n}{2}m_1 m_2,
\end{equation*}
which follows from \eqref{def:E} and \eqref{def:free_energy_functional_CW}. Thus
\begin{align*}
&M_1(\alpha_n) = 
G_{0,\beta}(m_g(\alpha_n),m_g(\alpha_n))
+\frac{\beta\alpha_n}{2}(m_g(\alpha_n))^2,\\
&M_2(\alpha_n) =
G_{0,\beta}(-m_l(\alpha_n),m_l(\alpha_n))
-\frac{\beta\alpha_n}{2}(m_l(\alpha_n))^2.
\end{align*}
Since $G_{0,\beta}$ is even in each coordinate, when expanding the arguments around $m^*$ we get
\[
G_{0,\beta}(m_g(\alpha_n),m_g(\alpha_n))
=
G_{0,\beta}(-m_l(\alpha_n),m_l(\alpha_n))+o(\alpha_n).
\]
Hence
\begin{equation*}
\varepsilon(\alpha_n) = M_1(\alpha_n) - M_2(\alpha_n)
=
\frac{\beta\alpha_n}{2}\Big[(m_g(\alpha_n))^2 + (m_l(\alpha_n))^2\Big]
+o(\alpha_n).
\end{equation*}
Using again the expansions, we obtain
\[
(m_l(\alpha_n))^2+(m_g(\alpha_n))^2
=
(m^*)^2+(m^*)^2 + o(1)
=
2(m^*)^2+o(1),
\]
hence
\[
\varepsilon(\alpha_n)
=
(m^*)^2 \beta \, \alpha_n + o(\alpha_n).
\]
\end{proof}

\subsubsection{Precise asymptotic of the partition function}
An important ingredient for the proof of Theorem \ref{thm_convergence_CW}, is the
derivation of a suitable asymptotic formula for the partition function $\bar{Z}_{n,\beta}$
in the non-uniqueness region. 

To state the result, let us first introduce some convenient notation.
For $\beta>0$, let $k$ be the number 
of maximizers of $G_{\alpha_n,\beta}$, i.e., 
\begin{align*}
    k=\begin{cases}
            1, \qquad & \text{ if } \beta \le \beta_c, \\
        2, \qquad & \text{ if } \beta  \in \left (\beta_c,\beta^*(\alpha_n)\right ], \\
        4, \qquad & \text{ if } \beta>\beta^*(\alpha_n).
    \end{cases}
\end{align*}
Depending on the region of $\beta$, and in particular on the number of maximizers $k$, we partition the magnetization space $\Gamma_{n/2}^2$ into disjoint subsets. For $\beta\neq \beta_c$, we choose $\delta \in (1/3, 1/2)$ and define
\begin{equation} \label{def-intorni}
\begin{split}
\mbox{-\,\, if } \beta< \beta_c, \,\,\,   &B_{0}= \left \{ \bm \in \Gamma_{n/2}^2: \, \|\bm-\bm^0\| \leq \frac{1}{n^\delta} \right \} \,\, \text{and}
    \,\, \mathcal{C}_0=\Gamma_{n/2}^2 \setminus B_0;\\
\mbox{-\,\, if } \beta> \beta_c, \,\,\,    & B_{i}= \left \{ \bm \in \Gamma_{n/2}^2: \, \|\bm-\bm^i\| \leq \frac{1}{n^\delta} \right \},\, i=1,\ldots, k \,\,\text{and}
    \,\, \mathcal{C}=\Gamma_{n/2}^2 \setminus \bigcup_{i=1}^k B_i.  
\end{split}
\end{equation}
Moreover, for $i=0,...,4$, we set
\begin{equation}
\label{def:DN}
R_i^{(n)}=\sqrt{n}\cdot B_i, \qquad 
D_i^{(n)}=\frac{1}{\sqrt{1- (m_1^i)^2} \sqrt{1- (m_2^i )^2}}\sum_{{\bf x} \in R_i^{(n)}} \frac{4}{\pi n} e^{-\frac{1}{2}{\bf x}^TQ_i^{(n)}{\bf x} },
\end{equation}
where $Q_i^{(n)}$  are positive defined matrices such that % being $\bm^i$ a maximizer
\begin{align*}
Q_i^{(n)}= -\frac{1}{2} \begin{pmatrix}
        \frac{\beta}{2}-\frac{1}{1-(m^i_1)^2} & \frac{\beta \alpha_n}{2} \\
        \frac{\beta \alpha_n}{2} & \frac{\beta}{2}-\frac{1}{1-(m^i_2)^2}
    \end{pmatrix}
    =-\frac 1 2 D^2G_{\alpha_n,\beta}(\bm^i)\,.
\end{align*}
Note that, the terms $D_i^{(n)}$ are uniformly bounded in $n$ and converge, for $n\to\infty$, to the constants 
\begin{align}\label{D-costanti}
    D_i= \lim_{n\to\infty} D_i^{(n)} = \frac{1}{4\pi \sqrt{s_1^i s_2^i}}
    \int_{\mathbb{R}^2} e^{ -\frac{1}{2} {\bf x}^T Q_i {\bf x}}d{\bf x}
 =    \frac{1}{\sqrt{ \left(1-\frac{\beta}{2}s_1^i\right)
\left(1-\frac{\beta}{2}s_2^i\right) -\frac{\beta^2\alpha^2}{4}s_1^is_2^i}}\,,
\end{align}
where, for all $i=1,\ldots,4$, we set  $s_1^i=1-(m_1^i)^2$, $s_2^i= 1-(m_2^i)^2$ and 
\begin{align*}
Q_i=\lim_{n\to\infty}   Q_i^{(n)}=     
-\frac{1}{2} 
\begin{pmatrix}
\frac{\beta}{2}-\frac{1}{s_1^i} & \frac{\beta \hat\alpha}{2}\\
\frac{\beta \hat\alpha}{2} & \frac{\beta}{2}-\frac{1}{s_2^i}
\end{pmatrix} 
\,.
\end{align*}
With this notation, we obtain the following result.
\begin{proposition}\label{Z_approx} 
For $n$ large enough, we have
\begin{equation*}
\bar{Z}_{n,\beta}=
\begin{cases}
e^{n\log2} D_0^{(n)}(1+o(1)), & \text{if }\beta<\beta_c, \\[5pt]
e^{\frac{n}{2} M_1(\alpha_n)} \sum_{i=1}^2 D_i^{(n)} (1+o(1)),& \text{if } \beta  \in \big(\beta_c,\beta^*(\alpha_n)\big], \\[5pt]
    e^{\frac{n}{2} M_1(\alpha_n)} \left (\sum_{i=1}^2 D_i^{(n)}+e^{-\frac{n}{2} \varepsilon(\alpha_n)} \sum_{i=3}^4 D_i^{(n)} \right ) (1+o(1)),  & \text{if } \beta>\beta^*(\alpha_n),
\end{cases} 
\end{equation*}
where $\varepsilon(\alpha_n)$ is defined in $\eqref{def:M1M2}$.
Moreover, if $\beta=\beta_c$, we get
\begin{equation}\label{Zcrit-bound}
\sqrt[4]{n} \, e^{n\log2}  D_{\text{low}}^{(n)}\,(1+o(1))\le \bar{Z}_{n,\beta_c}\le
\sqrt{n} \, e^{n\log2}  D_{\text{up}}^{(n)}\,(1+o(1)), 
\end{equation}
where $D_{\text{low}}^{(n)}$ and $D_{\text{up}}^{(n)}$ are explicit 
uniformly bounded function of $n$, converging to finite positive constants.
\end{proposition}

\begin{proof}
We begin by assuming $\beta>\beta_c$, which corresponds to the case of multiple maximizers, with $k\in\{2,4\}$. 
From  \eqref{def:hamiltonian_CW_magnet}, we can write
\begin{align*}
    \bar{Z}_{n,\beta} &=\sum_{\bm \in \Gamma_{n/2}^2} \binom{n/2}{\frac{1+m_1}{2}n/2} \binom{n/2}{\frac{1+m_2}{2}n/2} e^{\frac{\beta n}{2} E_{\alpha_n}(\bm)} = \sum_{i=1}^k \bar{Z}_{n,\beta}(B_i) +  \bar{Z}_{n,\beta}(\mathcal{C}),
\end{align*}
where we used the partition $\Gamma_{n/2}^2=\bigcup_{i=1}^k B_i \cup \mathcal{C}$ and the notation
$$\bar{Z}_{n,\beta}(A)= \sum_{\bm \in A} \sum_{\substack{\s\in\mathcal{X}_n:\\\bm(\s)=\bm}}e^{-\beta \bar{H}_n(\s)}\qquad \forall A\subseteq \Gamma_{n/2}^2\,.$$

We first analyze the contributions from $\bm\in B_i$, and without loss of generality we consider $i=1$. 
In $B_1$, the Stirling formula can be strengthened to the following asymptotical expression
\begin{align}\label{eq-coeffbin}
    \binom{n/2}{\frac{1+m_j}{2} n/2}= \frac{2}{\sqrt{\pi n (1-m_j^2)}}e^{- \frac{n}{2}\mathcal{I}(m_j)}(1+o(1)),
    \quad\mbox{for } j=1,2,
\end{align}
so that
\begin{align}\label{eq-Z-tipica}
    \bar{Z}_{n,\beta}(B_1)=\sum_{\bm \in B_1} \frac{4}{\pi n \sqrt{1-m_1^2} \sqrt{1-m_2^2}}e^{\frac{n}{2}G_{\alpha_n,\beta}(\bm)} (1+o(1)).
\end{align}
By Taylor expanding $G_{\alpha_n,\beta}(\bm)$ around the global maximizer $\bm^1=(m_g,m_g)$, we can write
\begin{align}\label{taylor_G}
    G_{\alpha_n,\beta}(\bm)= & \,\,G_{\alpha_n,\beta}(\bm^1)+\frac{1}{2} \left ( \frac{\beta}{2}-\frac{1}{1-m_g^2} \right )  \| \bm-\bm^1 \|^2+\frac{\beta \alpha_n}{2}(m_1-m_g)(m_2-m_g) \notag \\
    &+ O(\|\bm-\bm^1\|^3).
\end{align}
Substituting this term in \eqref{eq-Z-tipica}, and by the change of variable ${\bf{x}}=\sqrt{n}({\bf{m}}-{\bf{m}}^{1})$, we get
\begin{equation}\label{eq-Z-tipica2}
\begin{split}
    \bar{Z}_{n,\beta}(B_1)
    &=e^{\frac{n}{2}G_{\alpha_n,\beta}(\bm^1)}\sum_{{\bf x} \in R_1^{(n)}} \frac{4}{\pi n \sqrt{1- \left (\frac{x_1}{\sqrt{n}}+m_g \right )^2} \sqrt{1- \left (\frac{x_2}{\sqrt{n}}+m_g \right )^2}}  \\
    & \qquad\qquad\qquad\qquad\quad\,\,\, \cdot e^{ \frac{1}{4} \left ( \frac{\beta}{2}-\frac{1}{1-m_g^2} \right ) \|{\bf x}\|^2  +\frac{\beta \alpha_n}{4} x_1x_2}(1+o(1))  \\
    &= e^{\frac{n}{2}M_1(\alpha_n)} D_1^{(n)} (1+o(1)), 
\end{split}
\end{equation}
where we also used that, by definition \eqref{def-intorni}, $n\|\bm-\bm^1\|^3\le n^{1-3\delta}=o(1)$ if $\delta>1/3$. 
In the same way, we can provide analogous estimates of $Z_{n,\beta}(B_i)$, for $i=2,...,k$. 

We conclude by analyzing the contribution to $\bar{Z}_{n,\beta}$ coming from the configurations in $\mathcal{C}$. Suppose $k=4$ and divide $\mathcal{C}$ into the following four disjoint subsets:
\begin{equation}\label{4subset}
\begin{array}{ll}
\mathcal{C}_1= \{ \bm \in \mathcal{C} \, | \, m_1,m_2\ge 0\},\qquad& \mathcal{C}_2= \{ \bm \in \mathcal{C} \, | \, m_1,m_2<0\},\\ 
 \mathcal{C}_3= \{ \bm \in \mathcal{C} \, | \, m_1<0,m_2\ge 0\},\qquad& \mathcal{C}_4= \{ \bm \in \mathcal{C} \, | \, m_1\ge 0,m_2<0\}.
\end{array}
\end{equation}
We already stress that when $k=2$, the last two subsets are excluded from $\mathcal{C}$, but the remainder of the proof proceeds in the same way. We introduce the following notation for the local maxima of $G_{\alpha_n,\beta}$ on the above subsets:
\begin{equation*}
\begin{split}
    & g_1(\bm)=\max_{\bm \in \mathcal{C}_1} G_{\alpha_n,\beta}(\bm), 
    \qquad  g_2(\bm)=\max_{\bm \in \mathcal{C}_2}G_{\alpha_n,\beta}(\bm), \\
    & g_3(\bm)=\max_{\bm \in \mathcal{C}_3}G_{\alpha_n,\beta}(\bm), 
    \qquad g_4(\bm)=\max_{\bm \in \mathcal{C}_4}G_{\alpha_n,\beta}(\bm).
\end{split}
\end{equation*}
By the Stirling bound \eqref{StirBound}, recalling the notation introduced in \eqref{Max}, we get
\begin{align*}
    \bar{Z}_{n,\beta}(\mathcal{C}) 
    & \leq C n^2 \left (e^{\frac{n}{2}M_1(\alpha_n)} \sum_{i=1,2} e^{-\frac{n}{2}(M_1(\alpha_n)-g_i(\bm))}
    +e^{\frac{n}{2}M_2(\alpha_n)}\sum_{i=3,4}e^{-\frac{n}{2}(M_2(\alpha_n)-g_i(\bm))}
        \right ).
\end{align*}
Note that by construction, for all $i=1,\ldots, 4$ and for all $\bm \in \mathcal{C}_i$,
\begin{equation*}
n|M_i(\alpha_n)-g_i(\bm)|>n\|\bm-\bm^i\|^2>n^{1-2\delta}.
\end{equation*}
Hence we obtain 
\begin{align}\label{stat-atipica}
\bar{Z}_{n,\beta}(\mathcal{C})& \leq Cn^2e^{-\frac{1}{2}n^{1-2\delta}} (2e^{\frac{n}{2}M_1(\alpha_n)}+2e^{\frac{n}{2}M_2(\alpha_n)})<e^{\frac{n}{2}M_1(\alpha_n)} o(1)+e^{\frac{n}{2}M_2(\alpha_n)} o(1),
\end{align}
where in the last inequality, we use that $\delta<1/2$.
Together, \eqref{eq-Z-tipica2} and \eqref{stat-atipica} show that
\begin{align*}
        \bar{Z}_{n,\beta}=\left (e^{\frac{n}{2} M_1(\alpha_n)} \sum_{i=1}^2 D_i^{(n)}+\textbf{1}_{\{k>2\}}e^{\frac{n}{2}M_2(\alpha_n)} \sum_{i=3}^4 D_i^{(n)} \right ) (1+o(1)),
    \end{align*}
hence concluding the analysis for $k=2,4$.

A similar argument applies for $k=1$, whenever $\beta<\beta_c$.
In that case, the main contribution to the partition function comes 
from the neighborhood $B_0$ of the unique maximizer $\bm^0$ of $G_{\alpha_n,\beta}$, and the above procedure can be implemented even more directly. %\frac{2}{1+\alpha_n}.

In the critical regime $\beta=\beta_c$, when we still have the unique maximizer $\bm^0$, some extra care is needed to deal with the fact that
the Hessian of $G_{\alpha_n,\beta_c}$ in $\bm^0$  has a null eigenvalue.
In that case, a Taylor expansion around $\bm^0$ up to the fourth order is needed to deal with cases where the small order derivatives vanish as $n\to\infty$, depending on the value of $\alpha_n$. 
Explicitly, we get
\begin{align}\label{eq:Exp-G-Crit}
   G_{\alpha_n,\beta_c}(\mathbf{m})
&= 2\log 2
-\frac{\alpha_n}{2(1+\alpha_n)}(m_1-m_2)^2
- \frac{m_1^4+m_2^4}{12}
+ O(\|\mathbf{m}\|^5),
\end{align}
so that the following bounds are satisfied
\begin{equation}\label{bound-G-crit}
\left\{
\begin{array}{l}
G_{\alpha_n,\beta_c}(\mathbf{m}) 
\ge 2\log 2 -\frac{1}{4}(m_1-m_2)^2 - \frac{m_1^4+m_2^4}{12}
+ O(\|\mathbf{m}\|^5),\\[0.3cm]
G_{\alpha_n,\beta_c}(\mathbf{m}) 
\le 2\log 2 - \frac{m_1^4+m_2^4}{12}
+ O(\|\mathbf{m}\|^5)\,.
\end{array}
\right.
\end{equation}
As for the system outside criticality, we consider, for  $\delta_c \in (1/5, 1/4)$,
the partition of $\Gamma_{n/2}^2$
\begin{equation}\label{def:DNtilde}
\begin{split}
& B_c= \left \{ \bm \in \Gamma_{n/2}^2: \, \|\bm-\bm^0\| \leq \frac{1}{n^{\delta_c}} \right \} \,\, \text{and}
    \,\, \mathcal{\tilde{C}}=\Gamma_{n/2}^2 \setminus B_c \,. 
    \end{split}
\end{equation}
Moreover, we set
\begin{equation}
 \left\{\begin{array}{l}
 R_{\text{low}}^{(n)} =
\mathcal{O}_n B_c,
 \\[0.5cm]
 R_{\text{up}}^{(n)} =
n^{1/4} B_c,
\\
\end{array}
\right.
\qquad
\mbox{ where }
\qquad
\mathcal{O}_n=
\left(
\begin{array}{ll}
\frac{n^{1/4}}{2} & \frac{n^{1/4}}{2}\\
\frac{n^{1/2}}{2} & - \frac{n^{1/2}}{2}
\end{array}
\right)
,
\end{equation}
and define
\begin{equation}
\left\{
\begin{array}{l}
 D_{\text{low}}^{(n)}= 
\sum_{\textbf{x}\in R_{\text{low}}^{(n)}}
\frac{4}{\pi n^{5/4}}
e^{- \frac{1}{2}x_2^2
- \frac{1}{12}x_1^4},
\\[0.5cm]
D_{\text{up}}^{(n)}= 
\sum_{\textbf{x} \in R_{\text{up}}^{(n)}} \frac{4}{\pi n^{3/2}} e^{- \frac{1}{24}(x_1^4+x_2^4)}.
\end{array}
\right.
\end{equation}
Note that by construction 
\begin{equation*}
\left\{
\begin{array}{l}
\lim_{n\to\infty}  D_{\text{low}}^{(n)}= \frac{1}{\pi}\int_{\mathbb{R}^2}
e^{-\frac{1}{2}x_2^2 - \frac{1}{12}x_1^4} dx_1 dx_2 =
\frac{3^{1/4}}{\sqrt{\pi}} \Gamma\!\left(\frac14\right),\\[0.5cm]
\lim_{n\to\infty}  D_{\text{up}}^{(n)}= \frac{1}{4\pi}\int_{\mathbb{R}^2}
e^{-\frac{1}{24}(x_2^4 + x_1^4)} dx_1 dx_2 = \frac{\sqrt6}{8\pi}
\Gamma\!\left(\frac14\right)^2.
\end{array}
\right.
\end{equation*}

In order to prove \eqref{Zcrit-bound}, we start from the trivial observation that
$$
\bar{Z}_{n,\beta_c}(B_c) \le \bar{Z}_{n,\beta_c}=
\bar{Z}_{n,\beta_c}(B_c) +  \bar{Z}_{n,\beta_c}(\mathcal{\tilde{C}})\,.
$$
From the Stirling bound \eqref{StirBound}  and the upper bound on $G_{\alpha_n, \beta_c}$ in \eqref{bound-G-crit}, we get that
\begin{equation}\label{correzione}
    \bar{Z}_{n,\beta_c}(\mathcal{\tilde{C}}) 
     \leq C e^{n \log 2} \sum_{\bm\in \mathcal{\tilde{C}}} e^{-\frac{n}{2}(\frac{m_1^4+m_2^4}{12} + O(\|\bm\|^5))}= o(e^{n \log 2})\,,
    \end{equation}
where in the last step we used that, for all $\bm \in \mathcal{\tilde{C}}$
and since $\delta_c\in (1/5,1/4)$,
\begin{equation*}
n(m_1^4+m_2^4)\ge n^{1-4\delta_c}\ge n^{\tilde{\gamma}}\,, \mbox{ for some }
\tilde{\gamma}>0\,.
\end{equation*}

We now want to show that  $\bar{Z}_{n,\beta_c}(B_c)$ provides 
the correct exponential order $n\log 2$ of the partition function. Starting from the expression of $\bar{Z}_{n,\beta_c}(B_c)$ 
analog to \eqref{eq-Z-tipica}, which can be obtained in $B_c$
using the Stirling formula, and implementing the lower bound in \eqref{bound-G-crit}
and the change of variable ${\bf{x}}=n^{1/4}{\bf{m}}$, we obtain
\begin{equation*}
\begin{split}
    \bar{Z}_{n,\beta_c}(B_c)
    &\le e^{n\log 2} \sum_{\bm \in B_c} \frac{4}{\pi n \sqrt{1-m_1^2} \sqrt{1-m_2^2}}
    e^{-\frac{n}{2}\frac{m_1^4+m_2^4}{12}} (1+o(1))\\
    &=e^{n\log 2}\sum_{{\bf x} \in R_{\text{up}}^{(n)}} \frac{4}{\pi n} e^{-\frac{1}{24}(x_1^4+x_2^4)} (1+o(1)) \\
    &=\sqrt{n}\cdot e^{n\log 2} D_{\text{up}}^{(n)} (1+o(1)),
\end{split}
\end{equation*}
where we used that $n\|\bm\|^5\le n^{1-5\delta_c}=o(1)$. 
Together with \eqref{correzione}, this provides the upper bound in \eqref{Zcrit-bound}.

For the lower bound on  $\bar{Z}_{n,\beta_c}(B_c)$, and hence on $\bar{Z}_{n,\beta_c}$,
 we can analogously implement the upper bound in \eqref{bound-G-crit},  
 and the change of variable $\textbf{x}= \mathcal{O}_n\bm$, or explicitly
$$x_1= \frac{n^{1/4}}{2} (m_1+m_2),\qquad x_2= \frac{n^{1/2}}{2} (m_1-m_2)\,.$$
We then get
\begin{equation*}
\begin{split}
    \bar{Z}_{n,\beta_c}(B_c)
    &\ge e^{n\log 2} \sum_{\bm \in B_c} \frac{4}{\pi n \sqrt{1-m_1^2} \sqrt{1-m_2^2}}e^{-\frac{n}{2}\left(\frac{(m_1-m_2)^2}{2} + \frac{m_1^4+m_2^4}{12}\right)} (1+o(1))\\
    &=e^{n\log 2}\sum_{{\bf x} \in R_{\text{low}}^{(n)}} 
    \frac{4}{\pi n} e^{-\frac{1}{2}x_1^2+ \frac{1}{12}x_2^4} (1+o(1)) \\
    &=\sqrt[4]{n}\cdot e^{n\log 2} D_{\text{low}}^{(n)} (1+o(1)),
\end{split}
\end{equation*}
which concludes the proof of the lower bound in \eqref{Zcrit-bound} and 
of the proposition.
\end{proof}

\subsection{Proof of of Theorem \ref{thm_convergence_CW}}
Let us first assume that $\beta\le\beta_c$. In this regime,
by Proposition \ref{Z_approx}, the partition function is bounded below as
$$ \bar{Z}_{n,\beta}\ge C'\sqrt[4]{n} e^{n\log 2}\,, $$
for some constant $C'>0$ independent of $n$. 
Using this bound, together with the Stirling approximation \eqref{StirBound}, 
we get that  $\forall\, \varepsilon>0$,
\begin{equation}\label{ExpProb}
\begin{split}
\bar{\m}_{n,\beta}(\s\,:|\bm^{(n)}(\s)-\bm^0|>\varepsilon)
&\le C'' \sqrt[4]{n} \sum_{ \|\bm\|>\varepsilon} e^{\frac{n}{2}\left(G_{\alpha_n,\beta}(\bm)-2\log2\right)}\\
&\le 
C'' n^{9/4} e^{- c n \varepsilon^2}\le e^{- c'n\varepsilon^2} ,
\end{split}
\end{equation}
for $C''=C\cdot C'$ ($C$ being the constant appearing in \eqref{StirBound}) and  suitable constants $c,c'>0$. Statement \textit{(a)} of Theorem \ref{thm_convergence_CW} 
follows immediately from the above exponential convergence in probability.

We now assume that $\beta>\beta_c$, and treat in detail only the case where $\lim_{n\to\infty}\alpha_n=0$, the other being analogous. 
For notation convenience, let $\bar{\pi}_{n,\beta}$ be the measure on $\Gamma_{n/2}^2$ induced by $\bar{\mu}_{n,\beta}$ under the map $\mathcal{X}_{n}\to \Gamma_{n/2}^2$. More precisely,
\begin{align}\label{def_measureF_CW}
    &\bar{\pi}_{n,\beta} (\bm)=\sum_{\substack{\sigma \in \mathcal{X}_{n}: \\ \bm^{(n)}(\sigma)=\bm}} \bar{\mu}_{n,\beta}(\sigma).
\end{align}
To prove the convergence, it suffices to analyze, for any continuous bounded function $\varphi \in C_b(\mathbb{R}^2)$, 
\begin{align*}
   \lim_{n \to \infty} \sum_{\bm \in \G_{n/2}^2 } \varphi(\bm) \bar{\pi}_{n,\beta}(\bm)\,.
\end{align*}
By Proposition \ref{proposition:minima1} and Lemma \ref{lem:critical0}, 
the function $G_{\alpha_n,\beta}$ presents, for all $n$ large enough and $\beta>2$, four maximizers. 
With the same notation as introduced in \eqref{def-intorni}, we can write
\begin{align}\label{eq:partition_sum}
    \sum_{\bm \in \Gamma_{n/2}^2} \varphi(\bm) \bar{\pi}_{n,\beta}(\bm)=\sum_{i=1}^4\sum_{\bm \in B_i} \varphi(\bm) \bar{\pi}_{n,\beta}(\bm)+\sum_{\bm \in \mathcal{C}} \varphi(\bm) \bar{\pi}_{n,\beta}(\bm).
\end{align}

We begin by analyzing the sum over $B_1$ in \eqref{eq:partition_sum}. By Proposition \ref{Z_approx} and the Stirling formula
\eqref{eq-coeffbin}, it holds
\begin{align*}
    \sum_{\bm \in B_1} \varphi(\bm) \bar{\pi}_{n,\beta}(\bm)&
    =\sum_{\bm \in B_1} \frac{4}{\pi n} \varphi(\bm)\frac{e^{ \frac{n}{2} (G_{\alpha_n,\beta}(\bm)-M_1(\alpha_n))}}{\sqrt{1-m_1^2}\sqrt{1-m_2^2}} \\
& \qquad\qquad \cdot   \frac{1}{ \sum_{i=1}^2 D_i^{(n)}+e^{-\frac{n}{2} \varepsilon(\alpha_n)} \sum_{i=3}^4 D_i^{(n)} } (1+o(1)).
    \end{align*}
From the Taylor expansion of $G_{\alpha_n,\beta}(\bm)$ around $\bm^1$ and the change of variable ${\bf x}=\sqrt{n} (\bm-\bm^1)$, we get
\begin{equation}\label{passo}
\begin{split}
    \sum_{\bm \in B_1} \varphi(\bm) \bar{\pi}_{n,\beta}(\bm)
    &=
    \sum_{\textbf{x} \in R_1^{(n)}} \frac{4}{\pi n}\varphi \left (\frac{{\bf x}}{\sqrt{n}}+\bm^1 \right )
    \frac{e^{-\frac{1}{2}{\bf x}^TQ_1^{(n)}{\bf x} }}{\sqrt{1- m_g^2} \sqrt{1- m_g^2}}\\
    &\hspace{2cm}\cdot \frac{1}{\sum_{i=1}^2 D_i^{(n)}+e^{-\frac{n}{2} \varepsilon(\alpha_n)} \sum_{i=3}^4 D_i^{(n)}}
     (1+o(1))\,,
\end{split}
\end{equation}
From the limit \eqref{D-costanti} taken with $\hat\alpha=0$,  together with Lemma \ref{lem:epsilonalpha}, 
it holds that
\begin{align}\label{eq:limitD}
   \lim_{n \to \infty} \left ( \sum_{i=1}^2 D_i^{(n)}+e^{-\frac{n}{2} \varepsilon(\alpha_n)} \sum_{i=3}^4 D_i^{(n)} \right )=
   \begin{cases}
   2D_1,\qquad & \text{ if } \alpha_n \gg 1/n, \\
   2D_1(1+e^{-c}), \qquad & \text{ if } \alpha_n=c/n,\\
     4D_1, \qquad & \text{ if } \alpha_n \ll 1/n\,.
\end{cases}
\end{align}
Taking the limit $n\to\infty$ in \eqref{passo}, we then get
\begin{align*}
      \lim_{n \to \infty} \sum_{\bm \in B_1} \varphi(\bm) \bar{\pi}_{n,\beta}(\bm)=
      \begin{cases}
          \frac{1}{2} \varphi(\bm^1),  & \qquad \text{ if } \alpha_n \gg 1/n, \\
          \frac{1}{2(1+e^{-c})} \varphi(\bm^1),  & \qquad \text{ if } \alpha_n = c/n,\\
          \frac{1}{4} \varphi(\bm^1),  & \qquad \text{ if } \alpha_n \ll 1/n.
      \end{cases}
\end{align*}
Arguing as above, we obtain the same result for the sum over $B_2$ in \eqref{eq:partition_sum}.

Next, we analyze the sum over $B_3$ in \eqref{eq:partition_sum}. With an argument similar to the previous one, and recalling $M_1(\alpha_n)=M_2(\alpha_n)+\varepsilon(\alpha_n)$, 
we can write
\begin{align*}
     \sum_{\bm \in B_3}  \varphi(\bm) \bar{\pi}_{n,\beta}(\bm)
     &= e^{-\frac{n}{2} \varepsilon(\alpha_n)}\sum_{\bm \in B_3} \frac{4}{\pi n } \varphi(\bm) 
     \frac{e^{ \frac{n}{2} (G_{\alpha_n,\beta}(\bm)-M_2(\alpha_n))}}{ \sqrt{1-m_1^2}\sqrt{1-m_2^2}}\\
     &\qquad \qquad \qquad \qquad \cdot 
    \frac{1}{ \sum_{i=1}^2 D_i^{(n)}+e^{-\frac{n}{2} \varepsilon(\alpha_n)} \sum_{i=3}^4 D_i^{(n)} } (1+o(1)).
\end{align*}
From the Taylor expansion of $G_{\alpha_n,\beta}(\bm)$ around $\bm^3$ and the change of variable ${\bf x}=\sqrt{n} (\bm-\bm^3)$, we get
\begin{align*}
    \sum_{\bm \in B_3} \varphi(\bm) \bar{\pi}_{n,\beta}(\bm)
    =\, e^{-\frac{n}{2} \varepsilon(\alpha_n)} 
    &\sum_{\textbf{x} \in R_3^{(n)} } \frac{4}{\pi n }  \varphi \left(\frac{{\bf x}}{\sqrt{n}}+\bm^3 \right) 
   \frac{e^{-\frac{1}{2}{\bf x}^TQ_3^{(n)}{\bf x}}}{\sqrt{1- m_l^2} \sqrt{1- m_l^2}} 
    \\
    & \qquad \qquad\cdot\frac{1}{\sum_{i=1}^2 D_i^{(n)}+e^{-\frac{n}{2} \varepsilon(\alpha_n)} \sum_{i=3}^4 D_i^{(n)}} (1+o(1)),
\end{align*}
which together with \eqref{eq:limitD} and Lemma \ref{lem:epsilonalpha} implies that
\begin{align*}
   \lim_{n \to \infty} \sum_{\bm \in B_3} \varphi(\bm) \bar{\pi}_{n,\beta}(\bm)=
         \begin{cases}
          0,  & \qquad \text{ if }  \alpha_n \gg 1/n, \\
        \frac{e^{-c}}{2(1+e^{-c})} \varphi(\bm^3),  & \qquad \text{ if } \alpha_n = c/n,\\
          \frac{1}{4} \varphi(\bm^3),  & \qquad \text{ if }  \alpha_n \ll 1/n.
      \end{cases}
\end{align*}
Arguing as above, we obtain the same result for the sum over $B_4$ in \eqref{eq:partition_sum}.

Finally, we analyze the last sum in \eqref{eq:partition_sum}. Since $\varphi$ is bounded, we have
\begin{align*}
    &\sum_{\bm \in \mathcal{C}} \varphi(\bm) \bar{\pi}_{n,\beta}(\bm) \leq ||\varphi||_{\infty}\sum_{\bm \in \mathcal{C}} \bar{\pi}_{n,\beta}(\bm).
    \end{align*}
Using the partition of $\mathcal{C}$ given in \eqref{4subset}, 
we obtain four analogous terms. Let us first consider $\mathcal{C}_1$ and note that, from the Stirling bound \eqref{StirBound} and Lemma \ref{lem:FN_FNCW}, we get
\begin{align*}
    \sum_{ \bm \in \mathcal{C}_1} \bar{\pi}_{n,\beta}(\bm) 
    =\frac{C}{\bar{Z}_{n,\beta}}\sum_{\substack{\bm \in \mathcal{C}_1}}
   e^{ \frac{n}{2} G_{\alpha_n,\beta}(\bm)}
        \leq
    \frac{Ce^{\frac{n}{2} M_1(\alpha_n)}}{\bar{Z}_{n,\beta}}\sum_{\bm \in \mathcal{C}_1}e^{ -\frac{n}{2} (M_1(\alpha_n)-G_{\alpha_n,\beta}(\bm))},
\end{align*}
for some constant $C>0$. Since $\bm \in \mathcal{C}$ and $\delta <1/2$, we have that,
for some $\tilde C>0$,
$n|M_1(\alpha_n)-G_{\alpha_n,\beta}(\bm)| \geq \tilde C n\|\bm^1-\bm\|^2 \geq \tilde C n^{1-2\delta}=o(1)$. Thus,
\begin{align*}
    \sum_{\bm \in \mathcal{C}_1} \bar{\pi}_{n,\beta}(\bm) \leq\frac{e^{\frac{n}{2} M_1(\alpha_n)}}{\bar{Z}_{n,\beta}} o(1)\,.
\end{align*}
The same holds when considering the remaining sums over $\mathcal{C}_2$, $\mathcal{C}_3$, and $\mathcal{C}_4$, concluding  
the proof for the regime $\beta>2$, with $\lim_{n\to\infty}\alpha_n=0$.
\qed

%%%%%%%%%%%%%%%%%%%%%%%%%%%%%%%%%%%%%%%%%%%%%%%%%%%%%%%%%%%%%%%%%%%%%%%%%%%%%%%%%%%%%%%%%%%%%%%%%%%%%%%%%%%%%%%%%%%%%%

\section{Typical configurations and concentration of the random Gibbs measure}\label{sec:concentration}

In this section, we present four lemmas needed for the proof of Theorem \ref{thm_convergence} on the magnetization law of the Ising model on the 2-SBM, and then provide the proof of the theorem. In particular, in Section \ref{subsec:typical} we introduce a suitable notion of typical configurations and establish concentration estimates for the random environment, showing that these properties hold with overwhelming probability uniformly over all configurations. In Section \ref{subsec:concentration}, we compare the Hamiltonian of the model with its mean-field counterpart and show that their difference is uniformly negligible. This allows us to transfer the asymptotic behavior of the magnetization measure 
from the bipartite CW model to Ising model on the 2-SBM, and then complete the proof of Theorem \ref{thm_convergence}, 
which is given in Section \ref{subsec:convergence_mag}.

\subsection{Typical configurations}\label{subsec:typical}
We begin by introducing a convenient partition of the set of edges based on the alignment of spins. For $\sigma \in \mathcal{X}_n$, define
\begin{equation*}
\begin{split}
&\cE_{I}^{\pm}(\sigma)=\{(i,j) \in \cE_I \mid \sigma_{i} \sigma_{j}=\pm 1\}, \\
&\cE_{E}^{\pm}(\sigma)=\{(i,j) \in \cE_E \mid \sigma_{i} \sigma_{j}=\pm 1\}.
\end{split}
\end{equation*}
We compute the cardinalities of these sets in terms of magnetization. For $j=1,2$, let $r_j$ and $s_j$ denote the numbers of vertices in $V_j$ with spins $+1$ and $-1$, respectively, in a configuration $\sigma$. Setting $\ell=\frac{n}{2}$, and dropping for notational convenience
the dependence  on $\s$ and $n$, we have $r_j+s_j=\ell$ and $r_j-s_j=\ell m_j$, so that
\[
r_j=\ell\frac{1+m_j}{2}, \qquad s_j=\ell\frac{1-m_j}{2}.
\]
A direct computation then gives
\begin{equation} \label{cardinality_partial}
\begin{aligned}
    &|\cE_{I}^{+}(\sigma)|=\sum_{j=1}^2 \binom{r_{j}}{2}+\binom{s_j}{2}=\frac{\ell^2}{2} \left ( 1+\frac{\|\bm\|^2}{2} \right )-\ell, \\ 
    &|\cE_{E}^{+}(\sigma)|=r_1r_2+s_1s_2=\frac{\ell^2}{2}(1+m_1m_2).
\end{aligned}
\end{equation}
We further denote by $\cE^{+}(\sigma)=\cE_{I}^{+}(\sigma)\, \cup \, \cE_{E}^{+}(\sigma)$ the set of aligned pairs, and by $\cE^{-}(\sigma)$ its complement. With this notation, the Hamiltonian can be rewritten as
\begin{align}\label{hamiltonian_ER_magnet}
H_n(\sigma)= -\frac{1}{np} \left\{\sum_{(i, j)\in \cE^{+}(\sigma)} (\xi_{i j}+\zeta_{i j})  - \sum_{(i, j)\in \cE^{-}(\sigma)} (\xi_{i j}+\zeta_{i j}) \right \}.
\end{align}

We now introduce a set of typical environments where the random sums appearing above are well concentrated around their expectations. Define
\begin{align}\label{set_concentration}
    \mathcal{A}=\left\{\sigma \,: \left | \sum_{(i, j)\in \cE^{+}(\sigma)}\!\!\!\left(\xi_{ij}+\zeta_{ij}\right) -p\left(\left|\cE_{I}^{+}(\s)\right|+ \alpha_n \left| \cE_{E}^{+}(\s)\right|\right) \right |  \leq \rho_n \,p\left(\left|\cE_{I}^{+}(\s)\right|+\alpha_n\left|\cE_{E}^{+}(\s)\right|\right)\right\}
\end{align}
where $\rho_n>0$ is such that $\rho_n \to 0$ and $\rho_n \geq \frac{3}{\sqrt{pn}}$, and set
\begin{align}\label{Omega-good}
\bar{\Omega}_n=\bigcap_{\sigma \in \mathcal{X}_n} \left\{\omega \in \Omega:\, \sigma \in \mathcal{A}\right\} \subset \Omega.
\end{align}
Note that the definition of $\bar{\Omega}_n$ only involves the contribution from $\cE^{+}(\sigma)$. However, this is sufficient to control also the contribution from $\cE^{-}(\sigma)$. Indeed, the total sum $\sum_{(i,j)\in E} (\xi_{ij}+\zeta_{ij})$ is a sum of independent Bernoulli random variables and therefore concentrates around its expectation with overwhelming probability by Chernoff bound. Since this quantity does not depend on $\sigma$, such concentration holds uniformly over all configurations. Moreover,
\begin{equation}\label{eq:split}
\sum_{(i,j)\in \cE^{-}(\sigma)} (\xi_{ij}+\zeta_{ij})
= \sum_{(i,j)\in E} (\xi_{ij}+\zeta_{ij})
- \sum_{(i,j)\in \cE^{+}(\sigma)} (\xi_{ij}+\zeta_{ij}),
\end{equation}
which implies that an analogous bound holds for $\cE^{-}(\sigma)$ uniformly in $\sigma$. 

The next lemma shows that the sequence of events $(\bar{\Omega}_n)_n$ is asymptotically typical, in the sense that it eventually occurs with probability one.

%%%%%%%%%%%%%%%%%%%%%%%%%%%%%%%%%%%%%%%%%
\begin{lemma}\label{lemma1}
Under assumptions \eqref{assump},
$$
\mathbb{P}\left(\bar{\Omega}^{c}_n\right) \leq c_{0} n \left(e^{-n c_{1}/2}+e^{-n c_{2}/2}\right),
$$
with $c_{0}, c_{1}, c_{2}>0$ constants.
\end{lemma}

\begin{proof}
By definitions \eqref{set_concentration} and \eqref{Omega-good}, we can write $\bar{\Omega}^{c}_n$ as
\begin{align*}
    \bigcup_{\sigma \in \mathcal{X}_n} \left \{\omega \in \Omega: \left | \sum_{(i, j)\in \cE^{+}(\sigma) } \!\!\!\!\!\!(\xi_{i j}+\zeta_{i j}) 
     - p(|\cE_{I}^{+}(\s)|+ \alpha_n | \cE_{E}^{+}(\s) |) \right| \geq \rho_n p(|\cE_{I}^{+}(\s)|+ \alpha_n |\cE_{E}^{+}(\s)|) \right \}.
\end{align*}
Hence,
\begin{equation}
\label{eq:omegabarc}
\begin{split}
\mathbb{P}\left(\bar{\Omega}^{c}_n\right) &\leq \sum_\sigma \mathbb{P} \left ( \sum_{(i, j)\in \cE^{+}(\sigma) } (\xi_{i j}+\zeta_{i j}) 
\geq p(1+\rho_n) (|\cE_{I}^{+}(\s)|+ \alpha_n |\cE_{E}^{+}(\s)|) \right ) \\
& \quad + \sum_\sigma \mathbb{P}\left ( \sum_{(i, j)\in \cE^{+}(\sigma) }(\xi_{i j}+\zeta_{i j}) \leq p(1-\rho_n) (|\cE_{I}^{+}(\s)|+ \alpha_n |\cE_{E}^{+}(\s)|) \right ).
\end{split}
\end{equation}
Each addend in the first sum in \eqref{eq:omegabarc} can be bounded by the exponential Markov inequality as
\begin{equation}
\label{eq:firstaddend}
\begin{split}
&\mathbb{P} \left(
\sum_{(i, j)\in \cE^{+}(\sigma) }(\xi_{i j}+\zeta_{i j}) \geq (1+\rho_n)p (|\cE_{I}^{+}(\s)|+ \alpha_n | \cE_{E}^{+} (\s)|) \right ) \\
&\le e^{ -\sup_{t \geq 0} \left \{ (\left|\cE_{I}^{+}(\s)\right|+ \alpha_n \left|\cE_{E}^{+}(\s)\right|)(1+\rho_n) t p- \left|\cE_{I}^{+} (\s)\right|\log \left(1-p+p e^{t}\right)-\left|\cE_{E}^{+}(\s)\right| \log \left(1-\alpha_n p+\alpha_n p e^{t}\right)\right \}},
\end{split}
\end{equation}
since a Bernoulli random variable $X$ with parameter $p$ satisfies $\mathbb{E}\left[e^{t X}\right]=1-p+p e^{t}$. Moreover, for each $t \geq 0$ and $\alpha_n, p \in [0,1]$, by the concavity of the logarithm we have that $$
\log \left(1-p \alpha_n +p \alpha_n e^{t}\right) \geq \alpha_n \log \left(1-p+p e^{t}\right).
$$
We then get that \eqref{eq:firstaddend} can be bounded by
\begin{equation*}
e^{ - (\left|\cE_{I}^{+} (\s)\right|+ \alpha_n \left|\cE_{E}^{+} (\s)\right|)\sup_{t \geq 0} \left \{ (1+\rho_n) t p-\log \left(1-p+p e^{t}\right) \right \} }\leq e^{ - (\left|\cE_{I}^{+}(\s)\right|+ \alpha_n \left|\cE_{E}^{+} (\s)\right|) \mathscr{I}_p (p(1+\rho_n))},
\end{equation*}
where
\begin{align*}
    \mathscr{I}_y(x)=x \log \left ( \frac{x}{y} \right ) + (1-x) \log \left (\frac{1-x}{1-y} \right ).
\end{align*}
Analogously, each addend in the second sum in \eqref{eq:omegabarc} can be bounded as
\begin{equation*}
\sum_{(i, j)\in \cE^{+}(\sigma) }(\xi_{i j}+\zeta_{i j}) \leq (1-\rho_n)p \left(|\cE_{I}^{+} (\s)|+ \alpha_n | \cE_{E}^{+} (\s)|) \right ) \leq e^{ - (\left|\cE_{I}^{+} (\s)\right|+ \alpha_n \left|\cE_{E}^{+} (\s)\right|) \mathscr{I}_p (p(1-\rho_n))}.
\end{equation*}
Thus, using \eqref{cardinality_partial} and \eqref{StirBound}, there exists a constant $C>0$ such that
\begin{equation*}
\begin{split}
    \mathbb{P}\left(\bar{\Omega}^{c}_n\right) 
     \leq & \sum_{\bm} Ce^{-\frac{n}{2}\mathcal{I}(\bm)} e^{ -\frac{n^2}{4} \left ( 2E_{\alpha_n}(\bm)+1+\alpha_n\right ) \mathscr{I}_p (p(1+\rho_n))} \\
     & + \sum_{\bm} Ce^{-\frac{n}{2}\mathcal{I}(\bm)} e^{ -\frac{n^2}{4} \left ( 2E_{\alpha_n}(\bm)+1+\alpha_n\right ) \mathscr{I}_p (p(1-\rho_n))}.
\end{split}
\end{equation*}
Since $E_{\alpha_n}(\bm)$ and $\mathscr{I}_p (p(1-\rho_n))$ are positive convex functions, and $\mathcal{I}(\bm)$ is a convex function where the minimum is attained in $-2\log(2)$, we have
\begin{align*}
  e^{-\frac{n}{2} (\mathcal{I}(\bm)+ nE_{\alpha_n}(\bm) \mathscr{I}_p (p(1-\rho_n)))}   \leq 
  e^{n\log(2)}.
\end{align*}
Therefore, 
\begin{align*}
    \mathbb{P}\left(\bar{\Omega}^{c}_n\right)
    & \leq C\sum_{\bm} e^{-\frac{n}{2}(\frac{n}{2} (1+\alpha_n) \mathscr{I}_p (p(1+\rho_n))-2\log(2))} +
    C\sum_{\bm} e^{-\frac{n}{2}(\frac{n}{2} (1+\alpha_n) \mathscr{I}_p (p(1-\rho_n))-2\log(2))} \\
    & \leq \frac{Cn^2}{4} \left (e^{-\frac{n}{2}(\frac{n}{2} \mathscr{I}_p (p(1+\rho_n))-2\log(2))} +
    e^{-\frac{n}{2}(\frac{n}{2} \mathscr{I}_p (p(1-\rho_n))-2\log(2))} \right ),
\end{align*}
since $\alpha_n>0$ and $\bm \in \Gamma_{n/2}^2$.
When analyzing $\mathscr{I}_{p}(p(1 \pm \rho_n))$, by Taylor expansion, we have
\begin{equation*}
\begin{split}
\mathscr{I}_{p}(p(1+\rho_n)) 
&\geq  p\left(\rho_n+\frac{\rho_n^{2}}{2}-\frac{\rho_n^{3}}{6}\right)+(1-p)\left(-\frac{p}{1-p} \rho_n\right) \geq p \frac{\rho_n^2}{3},
\\
\mathscr{I}_{p}(p(1-\rho_n)) 
&\geq p\left(-\rho_n+\frac{\rho_n^{2}}{2}\right)+(1-p)\left(\frac{p}{1-p} \rho_n\right) =p \frac{\rho_n^2}{2}.
\end{split}
\end{equation*}
Therefore, there exist two constants $c_1,c_2>0$ such that
\begin{align*}
   &\frac{n}{2} \mathscr{I}_{p}(p(1+\rho_n))- 2\log 2 \geq \frac{n}{6} p \rho_n^2- 2\log 2\geq c_1, \\
   &\frac{n}{2} \mathscr{I}_{p}(p(1-\rho_n))- 2\log 2 \geq \frac{n}{4} p \rho_n^2- 2\log 2 \geq c_2,
\end{align*}
where the last inequalities follow by assuming that $\rho_n$ decreases to zero more slowly than $3/\sqrt{pn}$. Hence, we can conclude
\begin{align*}
   \mathbb{P}\left(\bar{\Omega}^{c}\right) 
&\leq \frac{Cn^2}{4}  \left ( e^{-nc_1/2} + e^{-n c_2/2} \right ).
\end{align*}
\end{proof}

As an application of the above result, we obtain the following characterization of typical configurations.
\begin{lemma}\label{lem:omega*}
Let $\Omega^{*}=\left\{\omega \in \Omega \mid \exists\, n_{0}:\, \omega \in \bar{\Omega}_n \,\text{ for all } \, n \geq n_{0}\right\}$.
Under assumptions \eqref{assump}, $\mathbb{P}(\Omega^{*})=1$.
\end{lemma}

\begin{proof}
Note that $\mathbb{P}\left(\Omega^{*}\right)=1-\mathbb{P}\left(\left(\Omega^{*}\right)^{c}\right)$ with 
$\left(\Omega^{*}\right)^{c}=\left\{\omega \in \Omega \mid \forall\, n_{0} \, \exists\, n \geq n_{0} \text{ s.t. } \omega \in \bar{\Omega}^{c}_n\right\}$.
Then
$0 \leq \mathbb{P}\left(\left(\Omega^{*}\right)^{c}\right)
=\mathbb{P} \left( \limsup_{n\to \infty} \bar{\Omega}^{c}_n\right)=0$,
where the last equality follows from Borel-Cantelli Theorem and by Lemma \eqref{lemma1}. 
\end{proof}

\subsection{Concentration of the random Gibbs measure}\label{subsec:concentration}
We now compare the Hamiltonian of the Ising model on the 2-SBM \eqref{def:hamil-DCW} 
with the Hamiltonian of the bipartite CW model \eqref{def:hamiltonian_CW}, its mean-field counterpart. For $\sigma \in \mathcal{X}_n$, let
\begin{align}\label{perturbation_hamiltonian}
    \Delta_n(\sigma)=H_n(\sigma)-\bar{H}_n(\sigma).
\end{align}
The following lemma  provides a bound on $\Delta_n(\sigma)$
for typical configurations $\s$.

\begin{lemma}\label{lem:Delta_n}
Under assumptions \eqref{assump}, for every $\omega \in \Omega^*$ and $\sigma\in \mathcal{X}^n$,
\begin{align*}
    |\Delta_n(\sigma)| \leq \frac 32 n \rho_n,
    \end{align*}
 with $\rho_n \to 0$ as $n \to \infty$ is defined in \eqref{set_concentration}.
\end{lemma}

\begin{proof}

By \eqref{hamiltonian_ER_magnet}, \eqref{eq:split} and \eqref{perturbation_hamiltonian}, we have
\begin{align*}
    |\Delta_n(\sigma)|&=|H_n(\sigma)-\bar{H}_n(\sigma)| \\
    &= 
    \Bigg| -\frac{1}{np}  \bigg (2\sum_{(i, j) \in \cE^{+}(\sigma)}(\xi_{i j}+\zeta_{i j}) 
    - \sum_{(i, j)\in E}(\xi_{i j}+\zeta_{i j}) \bigg) \\
    & \qquad +\frac{1}{n} \left (2\left(|\cE_{I}^{+} (\s)|+\alpha_n |\cE_{E}^{+} (\s)|\right)-\left(|\cE_{I}|+\alpha_n |\cE_{E}|\right)\right ) \Bigg| \\
    &\leq \frac{3\rho_n}{n}(|\cE_{I}(\s)|+ \alpha_n | \cE_{E}(\s) |),
    \end{align*}
where the inequality follows by the definition of $\Omega^*$, and in particular by the concentration  property described 
in \eqref{set_concentration}. Since $|\cE_I(\sigma)|=|\cE_E(\sigma)|= n^2/4$, and $\alpha_n\le 1$, we conclude
\begin{align*}
    |\Delta_n(\sigma)|& \leq \frac{3}{4} (1+\alpha_n)n\rho_n\leq \frac{3}{2}n \rho_n .
\end{align*}
\end{proof}

Next, let $\pi_{n,\beta}$ be the measure on $\Gamma_{n/2}^2$ induced by $\mu_{n,\beta}$ under the map $\sigma \to \bm^{(n)}(\s)$, namely
\begin{align}
\label{FN_ER}
    \pi_{n,\beta} (\bm)=\sum_{\substack{\sigma\in \mathcal{X}_{n}: \\ \bm^{(n)}(\sigma)=\bm}} \mu_{n,\beta}(\sigma).
\end{align}
Let $\bar{\pi}_{n,\beta}$ denote the analogous measure for the bipartite CW model, as defined in \eqref{def_measureF_CW}.

\begin{lemma}\label{lem:FN_FNCW}
Under assumptions \eqref{assump}, for all $\omega \in \Omega^*$ and all $\bm \in \Gamma_{n/2}^2$,
\[
    e^{-\frac{3}{2}n \rho_n} \bar{\pi}_{n,\beta}(\bm)\leq \pi_{n,\beta} (\bm)\leq e^{\frac{3}{2}n \rho_n} \bar{\pi}_{n,\beta} (\bm).
\]
\end{lemma}
\begin{proof}
By \eqref{FN_ER}, we can write
   \begin{align*}
       \pi_{n,\beta}(\bm)=\sum_{\substack{\sigma \in \mathcal{X}_n: \\ \bm^{(n)}(\sigma)=\bm}} \mu_{n,\beta}(\sigma)=
       \sum_{\substack{\sigma \in \mathcal{X}_n: \\ \bm^{(n)}(\sigma)=\bm}} \frac{e^{-\beta H_n(\sigma)}}{Z_{n,\beta}}=
       \sum_{\substack{\sigma \in \mathcal{X}_n: \\ \bm^{(n)}(\sigma)=\bm}} \frac{e^{-\beta (\Delta_n(\sigma)+\bar{H}_n(\sigma))}}{Z_{n,\beta}}.
   \end{align*}
Moreover, by Lemma \ref{lem:Delta_n}, we have
    \begin{align*}
    e^{- \frac{3}{2}\beta n \rho_n} \leq e^{-\beta \Delta_n(\sigma)} \leq e^{ \frac{3}{2}\beta n\rho_n},
\end{align*}
which readily provides the result.
\end{proof}

\subsection{Asymptotic magnetization law}\label{subsec:convergence_mag}

We are now ready to prove Theorem \ref{thm_convergence} and Corollary \ref{cor}.

\begin{proof}[Proof of Theorem \ref{thm_convergence}]
We start with part (i). For every $\omega \in \Omega^*$ and every $\sigma \in \mathcal{X}_n$, we can write
\begin{align*}
|f_{n,\beta} - \bar{f}_{n,\beta}|
&= \frac{1}{\beta n} \Bigg| \log \sum_{\sigma} e^{-\beta H_n(\sigma)} - \log \sum_{\sigma} e^{-\beta \bar{H}_n(\sigma)} \Bigg| = \frac{1}{\beta n} \Bigg| \log \frac{\sum_{\sigma} e^{-\beta \bar{H}_n(\sigma)} e^{-\beta \Delta_n(\sigma)}}{\sum_{\sigma} e^{-\beta \bar{H}_n(\sigma)}} \Bigg| \\
&\le \frac{1}{\beta n} \log \max_{\sigma} e^{\beta |\Delta_n(\sigma)|} \le \frac{1}{\beta n} \log e^{\frac{3}{2}\beta n \rho_n} =  \frac{3}{2}\rho_n,
\end{align*}
where in the last inequality we used Lemma \ref{lem:Delta_n}. Since $\rho_n \to 0$, the claim follows by Lemma \ref{lem:omega*}.

Parts (ii) and (iii) then follow from Lemma \ref{lem:FN_FNCW} and the corresponding results for the bipartite CW model (Theorem \ref{thm_convergence_CW} in Section \ref{sec:CW}) by a standard transfer argument: the measures $\pi_{n,\beta}$ and $\bar{\pi}_{n,\beta}$ are exponentially close, so any limiting property of $\bar{\pi}_{n,\beta}$ transfers to $\pi_{n,\beta}$.
\end{proof}

\begin{proof}[Proof of Corollary \ref{cor}] 
We start recalling the exponential convergence in probability derived in \eqref{ExpProb}, which holds under the bipartite CW measure in the whole uniqueness region.
By Lemmas \ref{lem:omega*} and \ref{lem:FN_FNCW}, the same holds $\mathbb{P}$-a.s. under the measure
$\mu_{n,\beta}$, for $n$ large enough. The stated strong law of large numbers then follows from the above exponential convergence as a straightforward application of the Borel-Cantelli Lemma. 
\end{proof}

%%%%%%%%%%%%%%%%%%%%%%%%%%%%%%%%%%%%%%%%%%%%%%%%%%%%%%%%%%%%%%%%%%%%%%%%%%%%%%%%%%%%%%%%%%%%%%%%%%%%%%%%%%

\section{Fluctuations of the magnetization at high temperatures}
\label{sec:fluctuations}

\subsection{Sketch of the proof}
To prove Theorem \ref{thm:CLT}, we use the fact that convergence of a sequence of probability measures is naturally defined via bounded continuous test functions. For $\varphi \in C_b(\mathbb{R}^2)$, define the generalized partition function
\begin{equation}
\label{def:ZN}
Z_{n,\beta}(\varphi) = \sum_{\sigma \in \mathcal{X}_n} 
e^{-\beta H_n(\sigma)} \, \varphi\left(\sqrt{n} \bm^{(n)}(\sigma) \right),
\end{equation}
and note that $Z_{n,\beta}(1) = Z_{n,\beta}$ is the standard partition function. Then
\begin{equation*}
%\int_{\mathbb{R}^2} \varphi(x) \, d\mathcal{L}(\sqrt{n} \bm)(x) = 
\mathbb{E}_{\mu_{n,\beta}}\left[\varphi\left(\sqrt{n}\bm^{(n)}\right)\right] =\frac{Z_{n,\beta}(\varphi)}{Z_{n,\beta}},
\end{equation*}
and the proof reduces to analyzing the limit of $Z_{n,\beta}(\varphi)/Z_{n,\beta}$ for all $\varphi \in C_b(\mathbb{R}^2)$ as $n \to \infty$. We show that this ratio converges in the limit to the expectation of $\varphi$ under an appropriate random vector. In order to do so, we define the tilted partition function
\begin{equation*}
\label{def:ZNtilted}
\widetilde{Z}_{n,\beta}(\varphi) = Z_{n,\beta}(\varphi) \, e^{-\log \cosh\left(\frac{\beta}{np}\right) \left(\sum_{(i,j) \in \cE_I} \xi_{ij} + \sum_{(i,j) \in \cE_E} \zeta_{ij}\right)},
\end{equation*}
and notice that 
\begin{equation*}
\frac{Z_{n,\beta}(\varphi)}{Z_{n,\beta}} = \frac{\widetilde{Z}_{n,\beta}(\varphi)}{\widetilde{Z}_{n,\beta}(1)} = \frac{\widetilde{Z}_{n,\beta}(\varphi)}{\E[\widetilde{Z}_{n,\beta}(\varphi)]} \frac{\E[\widetilde{Z}_{n,\beta}(\varphi)]}{\E[\widetilde{Z}_{n,\beta}(1)]} \frac{\E[\widetilde{Z}_{n,\beta}(1)]}{\widetilde{Z}_{n,\beta}(1)}.
\end{equation*}
We will prove that the first and third ratios in the r.h.s. converge to 1, while the second ratio converges to the desired limit. 

\subsection{Tilted partition function and preliminaries}

It is convenient to provide an equivalent representation of $\widetilde{Z}_n(\varphi)$ defined in \eqref{def:ZNtilted}.
In particular, we can write
\begin{equation}
\label{eq:ZT}
\widetilde{Z}_n(\varphi) = \sum_{\sigma \in \mathcal{X}_n} \varphi\left(\sqrt{n} \bm^{(n)}(\sigma) \right) T(\sigma),
\end{equation}
where
\begin{equation}
\label{def:Tsigma}
T(\sigma) = e^{\gamma \left( \sum_{(i,j) \in \cE_I} \sigma_i \sigma_j \xi_{ij} + \sum_{(i,j) \in \cE_E} \sigma_i \sigma_j \zeta_{ij} \right) - \log \cosh(\gamma) \left(\sum_{(i,j) \in \cE_I} \xi_{ij} + \sum_{(i,j) \in \cE_E} \zeta_{ij}\right)},
\end{equation}
with $\gamma = \frac{\beta}{np}$. With this notation, the following asymptotic estimates hold true.
\begin{lemma}
\label{lemma:T}
For $np \to \infty$ as $n \to \infty$, the two following statements hold.
\begin{itemize}
\item[(i)] For all $\sigma \in \mathcal{X}_n$,
$$ 
\E[T(\sigma)] = 
e^{- \frac{\beta^2}{8} (1+\alpha_n^2)-\frac{\beta}{2} + \left(\frac{\beta}{2}n+O\left(\frac{1}{np^2}\right)\right) E_{\alpha_n}(\bm^{(n)}(\s)) +O \left ( \frac{1}{n}\right )},
$$
where $E_{\alpha_n}(\bm)$ is defined in \eqref{def:hamiltonian_CW_magnet}.
\item[(ii)] For all $\sigma, \tau \in \mathcal{X}_n$,
\[
\begin{split}
\E[T(\sigma)T(\tau)] 
= &\, e^{- \frac{\beta^2}{4} (1+\alpha_n^2) -\beta
 + \left(\frac{\beta}{2}n+O\left(\frac{1}{np^2}\right)\right)
\Big( E_{\alpha_n} (\bm^{(n)}(\sigma)) +E_{\alpha_n}(\bm^{(n)}(\tau)) \Big)}\\
& \cdot e^{ O\left(\frac{1}{p}\right) E_{\alpha_n} (\bm^{(n)}(\sigma\tau))+ O \left ( \frac{1}{n}\right)}\,,
\end{split}
\]
where $\sigma \tau$ is the configuration obtained by multiplying the spins of $\sigma$ and $\tau$ component-wise, that is for each 
$i \in V$ we have $(\sigma \tau)_i=\sigma_i \tau_i$.
\end{itemize}

\end{lemma}

\begin{proof}
To prove the above lemma we need some technical preparation. 
For an integer $m$ and arbitrary complex variables $y$ and $z$, we introduce the function
\begin{equation*}
F_m(x,y,z) = \log \left(1-y+ye^{xz - m \log \cosh(z)}\right).
\end{equation*}
The power series expansions of some linear combinations of $F_m(x,y,z)$, 
with $m$ and $x$ fixed, have been investigated in \cite{KLS20}. 
In particular, by Lemma 2.1 in \cite{KLS20}, for $y$ and $z$ variables around $(0,0)$,
we have the following expansions:
\begin{equation}\label{lemma:Fcomb}
\begin{split}
& F_1(1,y,z) + F_1(1,y,-z) = -y^2 z^2 + O(y^2z^4), \\
& F_1(1,y,z) - F_1(1,y,-z) = 2yz + O(yz^3), \\
& F_2(2,y,z) - F_2(2,y,-z) = 4yz + O(yz^3), \\
& F_2(2,y,z) + F_2(2,y,-z) - 2F_2(0,y,z) = 4y(1-y) \tanh^2(z) + O(y^3z^3), \\
& F_2(2,y,z) + F_2(2,y,-z) + 2F_2(0,y,z) = -4y^2z^2 + O(y^2z^4).
\end{split}
\end{equation}

We prove the two statements separately. For part (i), we want to study
\begin{equation*}
\begin{split}
\E\left[T(\sigma) \right] & = \E\left[e^{ \gamma \left( \sum_{(i,j) \in \cE_I} \sigma_i \sigma_j \xi_{ij} + \sum_{(i,j) \in \cE_E} \sigma_i \sigma_j \zeta_{ij} \right) - \log \cosh(\gamma) \left( \sum_{(i,j) \in \cE_I} \xi_{ij} + \sum_{(i,j) \in \cE_E} \zeta_{ij} \right) }\right] \\
& = \prod_{(i,j) \in \cE_I} \E \left[ e^{\xi_{ij} \left( \gamma \sigma_i \sigma_j - \log \cosh(\gamma) \right) } \right] \prod_{(i,j) \in \cE_E} \E \left[ e^{ \zeta_{ij} \left( \gamma \sigma_i \sigma_j - \log \cosh(\gamma) \right)  } \right] \\
& = \prod_{(i,j) \in \cE_I}\!\!\! \left( 1 - p + pe^{\gamma \sigma_i \sigma_j - \log \cosh(\gamma)}\right) \prod_{(i,j) \in \cE_E} \!\!\!\left( 1 - \alpha_n p + \alpha_n p e^{\gamma \sigma_i \sigma_j - \log \cosh (\gamma)}\right).
\end{split}
\end{equation*}
Start by defining
\begin{align*}
& f(x) = \log \left( 1-p+pe^{\gamma x - \log \cosh(\gamma)}\right) = F_1(x,p,\gamma), \\
& f_{\alpha_n}(x) = \log \left( 1-\alpha_n p+\alpha_n pe^{\gamma x - \log \cosh(\gamma)}\right) = F_1(x, \alpha_n p,\gamma),
\end{align*}
so that we can write
\begin{equation*}
\E[T(\sigma)] = e^{ \sum_{(i,j) \in \cE_I} f(\sigma_i \sigma_j) + \sum_{(i,j) \in \cE_E} f_{\alpha_n}(\sigma_i \sigma_j)}.
\end{equation*}
Since $\sigma_i \in \{ -1,+1 \}$ for all $i$, we are only interested in $f(\pm 1)$ and $f_{\alpha_n}(\pm 1)$, which can be calculated by linearizing 
\begin{align*} 
\begin{array}{ll}
f(x) = a_0 + a_1 x, & x \in \{ -1,+1 \}, \\
f_{\alpha_n}(x) = a_{0, \alpha_n} + a_{1, \alpha_n} x, & x \in \{ -1,+1 \}.
\end{array}
\end{align*}
Using the above linearization, the expansions in \ref{lemma:Fcomb}, and substituting $\gamma = \frac{\beta}{np}$, we get
\begin{align*}
& a_0 = \frac{F_1(1, p, \gamma) + F_1(1, p, -\gamma)}{2} = -\frac{p^2 \gamma^2}{2}  + O(p^2 \gamma^4) = -\frac{\beta^2}{2n^2} + O\left(\frac{1}{n^4p^2} \right), \\
& a_1 = \frac{F_1(1, p, \gamma) - F_1(1, p, -\gamma)}{2} = p \gamma + O(p \gamma^3) = \frac{\beta}{n} + O\left(\frac{1}{n^3p^2} \right), \\
& a_{0, \alpha_n} = \frac{F_1(1, \alpha_n p, \gamma) + F_1(1, \alpha_n p, -\gamma)}{2} = -\frac{\alpha_n^2 p^2 \gamma^2}{2}  + O(p^2 \gamma^4) = - \frac{\alpha_n ^2 \beta^2}{2n^2} + O\left(\frac{1}{n^4p^2} \right), \\
& a_{1, \alpha_n} = \frac{F_1(1, \alpha_n p, \gamma) - F_1(1, \alpha_n p, -\gamma)}{2} = \alpha_n p \gamma + O(p \gamma^3) = \frac{\alpha_n \beta}{n} + O\left(\frac{1}{n^3p^2} \right).
\end{align*}
Hence, recalling that $m_i^{(n)}(\s)=\frac{1}{n/2} \sum_{j \in V_i} \sigma_j$ for $i=1,2$,
\begin{equation*}
\begin{split}
\E[T(\sigma)] & = e^{ \sum_{(i,j) \in \cE_I} f(\sigma_i \sigma_j) + \sum_{(i,j) \in \cE_E} f_{\alpha_n}(\sigma_i \sigma_j) } \\
& = e^{ \sum_{(i,j) \in \cE_I} (a_0 + a_1 \sigma_i \sigma_j) + \sum_{(i,j) \in \cE_E} (a_{0, \alpha_n} + a_{1, \alpha_n} \sigma_i \sigma_j)} \\
& = e^{\frac{n^2}{4}\left( a_0 + a_{0,\alpha_n}\right) + a_1 \left( \frac{n^2}{8} \|\bm^{(n)}(\s)\|^2 \right) + a_{1,\alpha_n} \frac{n^2}{4} m_1^{(n)}(\s) m_2^{(n)}(\s) -\frac{n}{2} (a_0+a_1)
} \\
& = e^{- \frac{\beta^2}{8} (1+\alpha_n^2)-\frac{\beta}{2} + \left(\frac{\beta}{2}n+O\left(\frac{1}{np^2}\right)\right) E_{\alpha_n}(\bm^{(n)}(\s)) +O \left(\frac{1}{n} \right )}.
\end{split}
\end{equation*}

For part (ii), we now want to study
\begin{equation*}
\begin{split}
\E\left[T(\sigma) T(\tau) \right] & = \prod_{(i,j) \in \cE_I} \E \left[ e^{\xi_{ij} \left( \gamma (\sigma_i \sigma_j + \tau_i \tau_j) - 2 \log \cosh(\gamma) \right) } \right] \\
&\qquad\qquad \cdot \prod_{(i,j) \in \cE_E} \E \left[ e^{ \zeta_{ij} \left( \gamma (\sigma_i \sigma_j + \tau_i \tau_j) - 2 \log \cosh(\gamma) \right)  } \right] \\
& = e^{\sum_{(i,j) \in \cE_I} g(\sigma_i \sigma_j + \tau_i \tau_j) + \sum_{(i,j) \in \cE_E} g_{\alpha_n}(\sigma_i \sigma_j + \tau_i \tau_j)},
\end{split}
\end{equation*}
where 
\begin{align*}
& g(x) = \log \left( 1-p+pe^{\gamma x - 2 \log \cosh(\gamma)}\right) = F_2(x,p,\gamma), \\
& g_{\alpha_n}(x) = \log \left( 1-\alpha_n p+\alpha_n pe^{\gamma x - 2 \log \cosh(\gamma)}\right) = F_2(x, \alpha_n p,\gamma).
\end{align*}
Since $\sigma_i \sigma_j + \tau_i \tau_j \in \{ -2, 0, 2 \}$ for all $i,j$, we are only interested in $g(0), g(\pm2), g_{\alpha_n}(0)$ and $g_{\alpha_n}(\pm2)$, which can be represented in the form
\begin{align*}
& g(x_1 + x_2) = b_0 + b_1 x_1 + b_2 x_2 + b_{12} x_1x_2, \qquad x_1, x_2 \in \{ \pm 1 \}, \\
& g_{\alpha_n}(x_1 + x_2) = b_{0, \alpha_n} + b_{1, \alpha_n} x_1 + b_{2, \alpha_n} x_2 + b_{12, \alpha_n} x_1x_2, \qquad x_1, x_2 \in \{ \pm 1 \}.
\end{align*}
Solving two appropriate systems of four linear equations in four variables each, using the expansions in \ref{lemma:Fcomb}, and substituting $\gamma = \frac{\beta}{np}$, we get
\begin{equation}\label{eq:cost-b}
\begin{split}
& b_0 =  \frac{F_2(2,p, \gamma) + F_2(2,p,-\gamma) + 2F_2(0,p,\gamma)}{4} =- \frac{\beta^2}{n^2} + O \left( \frac{1}{n^4p^2}\right), \\
& b_{12} = \frac{F_2(2,p, \gamma) + F_2(2,p,-\gamma) - 2F_2(0,p,\gamma)}{4} = O \left( \frac{1}{n^2p}\right), \\
& b_1 = b_2 = \frac{F_2(2,p, \gamma) - F_2(2,p,-\gamma)}{4} =\frac{\beta}{n} + O \left( \frac{1}{n^3 p^2}\right), \\
\end{split}
\end{equation}
and
\begin{equation}
\label{eq:cost-balpha}
\begin{split}
& b_{0, \alpha_n} = \frac{F_2(2,\alpha_n p, \gamma) + F_2(2,\alpha_n p,-\gamma) + 2F_2(0,\alpha_n p,\gamma)}{4} = - \frac{\alpha_n^2 \beta^2}{n^2} + O \left( \frac{1}{n^4p^2}\right), \\
& b_{12, \alpha_n} = \frac{F_2(2,\alpha_n p, \gamma) + F_2(2,\alpha_n p,-\gamma) - 2F_2(0,\alpha_n p,\gamma)}{4} = O \left( \frac{1}{n^2p}\right), \\
& b_{1,\alpha_n} = b_{2,\alpha_n} = \frac{F_2(2,\alpha_n p, \gamma) - F_2(2,\alpha_n p,-\gamma)}{4} = \frac{\alpha_n \beta}{n} + O \left( \frac{1}{n^3 p^2}\right).
\end{split}
\end{equation}
Hence we obtain
\begin{equation*}
\begin{split}
\E[T(\sigma)T(\tau)]  & =\, \exp \Bigg( \sum_{(i,j) \in \cE_I} g(\sigma_i \sigma_j + \tau_i \tau_j) + \sum_{(i,j) \in \cE_E} g_{\alpha_n}(\sigma_i \sigma_j + \tau_i \tau_j)\Bigg) \\
& = \, \exp \Bigg( \sum_{(i,j) \in \cE_I} (b_0 + b_1 \sigma_i \sigma_j + b_2 \tau_i \tau_j + b_{12} \sigma_i \sigma_j \tau_i \tau_j) \\
& \qquad \qquad \quad + \sum_{(i,j) \in \cE_E} (b_{0, \alpha_n} + b_{1, \alpha_n} \sigma_i \sigma_j + b_{2, \alpha_n} \tau_i \tau_j + b_{12, \alpha_n} \sigma_i \sigma_j \tau_i \tau_j) \Bigg).
\end{split}
\end{equation*}
Expanding the sums in the last display, and recalling the definition of $\sigma\tau$ in Lemma \ref{lemma:T}, we get 
\begin{equation*}
\begin{split}
\E[T(\sigma)T(\tau)] 
& 
= \, \exp \Bigg( \frac{n^2}{4}(b_0 + b_{0,\alpha_n})
 + b_1 \frac{n^2}{8} \|\bm^{(n)}(\sigma)\|^2  + b_{1,\alpha_n} \frac{n^2}{4} m_1^{(n)}(\sigma) m_2^{(n)}(\sigma) \\ 
& \qquad \qquad 
+ b_2\frac{n^2}{8} \|\bm^{(n)}(\tau)\|^2 + b_{1,\alpha_n} \frac{n^2}{4} m_1^{(n)}(\tau) m_2^{(n)}(\tau) \\ 
& \qquad \qquad + b_{12} \frac{n^2}{8} \|\bm^{(n)}(\sigma\tau)\|^2 + b_{12,\alpha_n} \frac{n^2}{4} m_1^{(n)}(\sigma\tau) m_2^{(n)}(\sigma\tau) \\
& \qquad \qquad -\frac{n}{2} (b_0+b_1+b_2+b_{12})\Bigg).
\end{split}
\end{equation*}
Finally, inserting the specific values in \eqref{eq:cost-b} and \eqref{eq:cost-balpha}, and developing as for the computation of $\E(T(\sigma))$, we obtain
\begin{equation*}
\begin{split}
\E[T(\sigma)T(\tau)] 
= &\, e^{- \frac{\beta^2}{4} (1+\alpha_n^2) -\beta
 + \left(\frac{\beta}{2}n+O\left(\frac{1}{np^2}\right)\right)
\Big( E_{\alpha_n} (\bm^{(n)}(\sigma)) +E_{\alpha_n} (\bm^{(n)}(\tau)) \Big)} \\
& \cdot e^{ O\left(\frac{1}{p}\right) E_{\alpha_n} (\bm^{(n)}(\sigma\tau))
+O\left ( \frac{1}{n} \right )}.
\end{split}
\end{equation*} 
\end{proof}

\subsection{Mean and variance of the tilted partition function}
We are now ready to prove the following results. 
\begin{lemma}[Mean of the tilted partition function]
\label{prop:meantilted} 
Let $\beta < \beta_c$. Under assumptions \eqref{assump}, for all non-negative $\varphi \in C_b(\R^2)$, $\varphi\not\equiv 0$, 
$$
\lim_{n\to\infty}\E[\widetilde{Z}_{n,\beta}(\varphi)] 2^{-n} = \frac{e^{ - \frac{\beta^2}{8} (1+\hat\alpha^2)-\frac{\beta}{2} }}{2} \sqrt{\det{(\Sigma)}} \, \E_X\left[ \varphi(X) \right], 
$$
where $\E_{X}$ denotes the expectation over $X \sim \mathcal{N}(0,\Sigma)$ and $\Sigma$ is given in \eqref{def:sigma}.
\end{lemma}

\begin{proof}
We start by defining a set of typical configurations whose magnetization is close to the unique minimum $\bm^0=(0,0)$ of the free energy functional, i.e., 
\begin{equation}
\label{def:typ_rn}
S_n = \left\{ \sigma \in \mathcal{X}_n : \|\bm^{(n)}(\sigma)\| \le r_n \right\}, \quad 
r_n = \sqrt{\frac{(np)^{\delta}}{n}}, \quad \delta \in (0,1/3) .
\end{equation}
The atypical configurations are then
\[
S_n^c = \mathcal{X}_n \setminus S_n = \left\{ \sigma \in \mathcal{X}_n : \|\bm^{(n)}(\sigma)\| > r_n \right\}.
\]
Recalling the definition of tilted partition function $\widetilde{Z}_n(\varphi)$ in \eqref{eq:ZT}, we can write
\begin{equation}
\label{eq:typicalatypical}
\begin{split}
\E [ \widetilde{Z}_n(\varphi)] & = \sum_{\sigma \in \mathcal{X}_n} \varphi\left(\sqrt{n} \bm^{(n)}(\sigma) \right) \E[T(\sigma)] \\
& = \sum_{\sigma \in S_n} \varphi\left(\sqrt{n} \bm^{(n)}(\sigma) \right) \E[T(\sigma)] + \sum_{\sigma \in S_n^c} \varphi\left(\sqrt{n} \bm^{(n)}(\sigma) \right) \E[T(\sigma)].
\end{split}
\end{equation}

\noindent
\underline{Step 1: Contribution of typical configurations.} Recall that Lemma~\ref{lemma:T}(i) gives
\[
\E[T(\sigma)] = \exp \left(- \frac{\beta^2}{8} (1+\alpha_n^2) -\frac{\beta}{2}+ \left(\frac{\beta}{2}n+O\left(\frac{1}{np^2}\right)\right) E_{\alpha_n}(\bm^{(n)}(\s)) +O \left (\frac{1}{n}\right ) \right ) ,
\]
Since $\|\bm(\sigma)\|^2 \le r_n^2$ for typical configurations, we get 
that $E_{\alpha_n}(\bm^{(n)}(\sigma)) \le r_n^2$, and hence
\begin{equation*}
O\left(\frac{1}{n p^2}\right) E_{\alpha_n}(\bm(\sigma)) \le  \frac{1}{n p^2}  r_n^2
= O\left( \frac{1}{(np)^{(2-\delta)}} \right) = o(1),
\end{equation*}
due to our choice of  $r_n$ and $\delta$ in \eqref{def:typ_rn}, and since $n p \to \infty$.
Moreover, using \eqref{eq-coeffbin}, the number of configurations with magnetization 
$\bm=(m_1,m_2)$, with $\|\bm\|\le r_n$, is given by
\begin{equation}
\label{eq:doublebinom}
  \binom{n/2}{\frac{1+m_1}{2} n/2} \binom{n/2}{\frac{1+m_2}{2} n/2}= 
   \frac{4}{\pi n\sqrt{(1-m_1^2)(1-m_2^2)}}e^{- \frac{n}{2}\mathcal{I}(\bm)}(1+o(1)),
\end{equation}
where $\mathcal{I}(\bm)=\mathcal{I}(m_1)+\mathcal{I}(m_2)$ is the entropic term introduced in \eqref{def:free_energy_functional_CW}.
Summing over magnetizations rather than configurations, the contribution of typical configurations becomes
\begin{equation}
\label{eq:typical_sum}
\begin{split} 
\sum_{\sigma \in S_n} \varphi(\sqrt{n} \bm^{(n)}(\sigma)) \E[T(\sigma)] &=C(\beta,\alpha_n)
\sum_{\substack{ \|\bm\|\le r_n}}
\frac{4}{\pi n}\varphi(\sqrt{n}\, \bm) \frac{e^{\Big( \frac{\beta}{2} n E_{\alpha_n}(\bm) -\frac{n}{2} \mathcal{I}(\bm) \Big)}}{\sqrt{(1-m_1^2)(1-m_2^2)}}
(1+o(1))\\
&= C(\beta,\alpha_n) \sum_{\substack{ \|\bm\|\le r_n}}
\frac{4}{\pi n}\varphi(\sqrt{n}\, \bm) \frac{e^{\Big( \frac{n}{2} G_{\alpha_n, \beta}(\bm)\Big)}}{\sqrt{(1-m_1^2)(1-m_2^2)}}(1+o(1))\,,
\end{split}
\end{equation}
with $C(\beta,\alpha_n)=e^{- \frac{\beta^2}{8} (1+\alpha_n^2) -\frac{\beta}{2}}$.
Elaborating as in \eqref{eq-Z-tipica}-\eqref{eq-Z-tipica2}, we first perform a Taylor expansion of $G_{\alpha_n,\beta}(\bm)$ around $G_{\alpha_n,\beta}(\bm^0)=2\log2$, getting 
\begin{align}\label{eq:G-sviluppo}
    G_{\alpha_n,\beta}(\bm)=2\log2-\frac{1}{2} \left ( 1-\frac{\beta}{2} \right ) \|\bm\|^2 +\frac{\beta \alpha_n}{2}m_1 m_2+ O(\|\bm\|^3),
\end{align}
and then we apply the change of variables ${\bf x}= \sqrt{n} \bm \in \R^2$. Altogether, using that for typical configurations $n\|\bm\|^3=o(1)$ and defining the $2\times2$ matrix 
\begin{equation}
\label{def:sigma_n}
\Sigma_n = \frac{2-\beta}{(1-\beta) + \frac{\beta^2(1-\alpha_n^2)}{4}} 
\begin{pmatrix} 
1 & \frac{\beta\alpha_n}{2-\beta} \\
\frac{\beta\alpha_n}{2-\beta} & 1
\end{pmatrix},
\end{equation}
we obtain 
$$
\sum_{\sigma \in S_n} \varphi(\sqrt{n} \bm^{(n)}(\sigma)) \E[T(\sigma)] =\frac{C(\beta,\alpha_n)}{4\pi} 2^n\sum_{ {\bf x} \in \sqrt{n}\cdot S_n}
\frac{16}{n}\varphi({\bf x})
\frac{ e^{-\frac{1}{2} ({\bf x})^\top \Sigma_n^{-1} {\bf x}}}{\sqrt{1-\frac{x_1^2}{n}}\sqrt{1- \frac{x_2^2}{n}}} (1+o(1)).
$$
We note that the Riemann sum in the last display can be approximated, for $n\gg 1$, by the integral
\begin{equation*}
\int_{\R^2} \varphi({\bf x})e^{ -\frac{1}{2} {\bf x}^\top \Sigma^{-1} {\bf x}} d{\bf x},
\end{equation*}
where the matrix $\Sigma = \lim_{n\to \infty} \Sigma_n$ is defined in \eqref{def:sigma}. Under the assumption $\beta<\beta_c$, $\Sigma$ is positive definite, 
which guarantees that the exponential weight is integrable over $\R^2$.
In particular, the integral can equivalently be expressed as a Gaussian expectation,
so that  altogether, for $X \sim \mathcal{N}(0, \Sigma)$, we obtain
\begin{equation}\label{eq:buca-finale}
\lim_{n \to \infty} \sum_{\sigma \in S_n} \varphi(\sqrt{n} \bm^{(n)}(\sigma)) \E[T(\sigma)] \, 2^{-n}
=
\frac{e^{- \frac{\beta^2}{8} (1+\hat\alpha^2) -\frac{\beta}{2}}}{2}  \sqrt{\det{(\Sigma)}} \, \E_X\left[ \varphi(X)\right].
\end{equation}
 
\noindent
\underline{Step 2: Contribution of atypical configurations.} 
Proceeding as in \eqref{eq:typical_sum}, and applying Lemma~\ref{lemma:T}(i), we easily get the bound
\begin{equation*}
\sum_{\sigma \in S_n^c} \varphi(\sqrt{n} \bm^{(n)}(\sigma)) \E[T(\sigma)] \le 2\|\varphi\|_\infty \!\!\!
\sum_{\substack{\bm \in \G_n\\ \|\bm\|> r_n}} \!\!\!\!
  \binom{n/2}{\frac{1+m_1}{2}n/2}\binom{n/2}{\frac{1+m_2}{2}n/2}e^{\left(\frac{\beta}{2}n+O\left(\frac{1}{np^2}\right)\right) E_{\alpha_n}(\bm (\sigma))},
\end{equation*}
where we used that $\|\varphi\|_\infty e^{- \frac{\beta^2(1+\alpha_n)}{8}-\frac{\beta}{2}+O\left(\frac{1}{n}\right)} \leq 2\|\varphi\|_\infty $. For large $n$, 
we then approximate the two binomials by using the local limit theorem
\begin{align}
\label{eq:binomialLLT}
    \binom{n/2}{\frac{1+m_1}{2}n/2} \binom{n/2}{\frac{1+m_2}{2}n/2}\leq \frac{C}{n} 2^ne^{ -\frac{n \| \bm \|^2}{4} }
\end{align}
where $C > 0$ is constant, and we obtain
\begin{align*}
        \sum_{\sigma \in S_n^c} \varphi(\sqrt{n} \bm^{(n)}(\sigma)) \E[T(\sigma)] 
        &\le 
        C'(\beta,\varphi) 2^n \sum_{\substack{\bm \in \G_n\\ \|\bm\|> r_n}} \frac{16}{n}
 e^{\frac{n}{2} \left( \beta E_{\alpha_n}(\bm) -\frac{\| \bm \|^2}{2}+O\left(\frac{\|\bm\|^2}{(np)^2}\right)\right)}\\
 & = C'(\beta,\varphi) 2^n \sum_{\substack{\bm \in \G_n\\ \|\bm\|> r_n}} 
 \frac{16}{n}
 e^{\frac{n}{2} \left( -\frac{1}{2}(1-\frac{\beta}{2})\|\bm\|^2 +\frac{\beta\alpha_n}{2}m_1m_2 +O\left(\frac{\|\bm\|^2}{(np)^2}\right)\right)},
 \end{align*}
for some constant $C'(\beta,\varphi)>0$.
Note that the main term in the exponent above corresponds to the main term in \eqref{eq:G-sviluppo}, which is the Taylor expansion of
$G_{\alpha_n,\beta}(\bm)$ around $G_{\alpha_n,\beta}(\bm^0)$. 
As before, applying the change of variables ${\bf x}=\sqrt{n}\bm$  
and letting $\Sigma_n$ be the matrix defined in \eqref{def:sigma_n}, 
the sum in the last display can be expressed as the Riemann sum
$$
\sum_{ {\bf x}\in \sqrt{n}\cdot S_{n}^c}
\frac{16}{n}
 e^{-\frac{1}{2} ({\bf x})^\top \Sigma_n^{-1} {\bf x} + O\left(\frac{\|{\bf x}\|^2}{(np)^2}\right)}
 =
 \sum_{ {\bf x}\in \sqrt{n}\cdot S_{n}^c}
\frac{16}{n}
 e^{-\frac{1}{2} ({\bf x})^\top \Sigma_n^{-1} {\bf x}}(1+o(1)).
$$
For $n\gg 1$, since $(np)^{\delta} \to \infty$ and the Gaussian tail outside the ball of radius $\sqrt{(np)^{\delta}}$ vanishes exponentially, the sum can be approximated by the integral
$$ \int_{\|{\bf x}\|^2> (np)^{\delta}}
 e^{-\frac{1}{2} ({\bf x})^\top \Sigma^{-1} {\bf x}}d{\bf x} =o(1),$$
where $\Sigma$ is defined in \eqref{def:sigma}. Altogether, we get that
\begin{equation}\label{eq:compl-finale}
\sum_{\sigma \in S_n^c} \varphi(\sqrt{n} \bm^{(n)}(\sigma)) \E[T(\sigma)] =o\left(2^n\right).
\end{equation}
Putting together \eqref{eq:typicalatypical}, \eqref{eq:buca-finale} and \eqref{eq:compl-finale}, we conclude.
\end{proof}

\begin{lemma}[Variance of the tilted partition function]
\label{prop:vartilted}
Let $\beta < \beta_c$. Under assumptions \eqref{assump}, for all non-negative $\varphi \in C_b(\R^2)$, $\varphi \not\equiv 0$, 
$$ \lim_{n \to \infty} \frac{{\rm Var}(\widetilde{Z}_{n,\beta}(\varphi))}{\E[\widetilde{Z}_{n,\beta}(\varphi)]^2} = 0. $$
As a consequence,
$$\frac{\widetilde{Z}_{n,\beta}(\varphi)}{\mathbb{E}[\widetilde{Z}_{n,\beta}(\varphi)]} \underset{n\to\infty}{\longrightarrow} 1 \quad \text{in } L^2(\Omega, \mathbb{P}).
$$
\end{lemma}
\begin{proof}
Similarly to the proof above, define a set of typical pairs of configurations as
\begin{equation*}
S'_n = \left\{ (\sigma, \tau) \in \mathcal{X}_{n}^2 : \|\bm^{(n)}(\sigma\tau)\|  \le r_n, k=1,2\right\},
\end{equation*}
with $r_n = \sqrt{\frac{(np)^{\delta}}{n}}, \, \delta \in (0,1/3)$, and denote the set of atypical pairs by
$(S_n')^c$. Recalling the definition of $\widetilde{Z}_n(\varphi)$ in \eqref{eq:ZT} and decomposing into a sum over typical and atypical pairs, we can write
\begin{equation}
\label{eq:vartypicalatypical}
\begin{split}
{\rm Var}\left(\widetilde{Z}_n(\varphi)\right) = &  \sum_{(\sigma, \tau) \in S'_n} \varphi\left(\sqrt{n} \bm^{(n)}(\sigma) \right) \varphi\left(\sqrt{n} \bm^{(n)}(\tau) \right) {\rm Cov}(T(\sigma), T(\tau)) \\
& + \sum_{(\sigma, \tau) \in (S'_n)^c} \varphi\left(\sqrt{n} \bm^{(n)}(\sigma) \right) \varphi\left(\sqrt{n} \bm^{(n)}(\tau) \right) {\rm Cov}(T(\sigma), T(\tau)).
\end{split}
\end{equation}

\noindent
\underline{Step 1: Contribution of typical configurations.} We start by observing that from the definition of the covariance and Lemma \ref{lemma:T} 
\begin{equation*}
\begin{split}
\mathrm{Cov}(T(\sigma),T(\tau)) & = \E[T(\sigma)T(\tau)] - \E[T(\sigma)]\E[T(\tau)] \\
& = \E[T(\sigma)] \E[T(\tau)] \left( e^{ O\left(\frac{1}{p}\right) E_{\alpha_n}(\bm^{(n)}(\sigma\tau))} - 1 \right).
\end{split}
\end{equation*}
For typical pairs $(\sigma, \tau) \in S'_n$ the exponent becomes $o(1)$, hence
\begin{equation*}
{\rm Cov}(T(\sigma), T(\tau)) = o\left( \E[T(\sigma)] \E[T(\tau)] \right).
\end{equation*}
We can then write
\begin{equation}
\label{eq:typicalpairs}
\begin{split}
&\sum_{(\sigma, \tau) \in S'_n} \varphi\left(\sqrt{n} \bm^{(n)}(\sigma) \right) \varphi\left(\sqrt{n} \bm^{(n)}(\tau) \right) {\rm Cov}(T(\sigma), T(\tau)) \\
& \qquad =\, o \left( \sum_{(\sigma, \tau) \in S'_n} \varphi\left(\sqrt{n} \bm^{(n)}(\sigma) \right) \varphi\left(\sqrt{n} \bm^{(n)}(\tau) \right) \E[T(\sigma)] \E[T(\tau)] \right) \\
& \qquad =\, o \left( \sum_{(\sigma, \tau) \in \mathcal{X}_n^2} \varphi\left(\sqrt{n} \bm^{(n)}(\sigma) \right) \varphi\left(\sqrt{n} \bm^{(n)}(\tau) \right)  \E[T(\sigma)] \E[T(\tau)] \right) \\
& \qquad =\, o\left( \E[\widetilde{Z}_n(\varphi)]^2 \right),
\end{split}
\end{equation}
where in the second equality we used that $\varphi \geq 0$. \\

\noindent
\underline{Step 2: Contribution of atypical configurations.} 
We write
\begin{equation}
\label{eq:atypicalpairterms}
\begin{split}
& \sum_{(\sigma, \tau) \in (S'_n)^c} \varphi\left(\sqrt{n} \bm^{(n)}(\sigma) \right) \varphi\left(\sqrt{n} \bm^{(n)}(\tau) \right) {\rm Cov}(T(\sigma), T(\tau))  \\
&\qquad \qquad \leq \, \sum_{(\sigma, \tau) \in (S'_n)^c} \varphi\left(\sqrt{n} \bm^{(n)}(\sigma) \right) \varphi\left(\sqrt{n} \bm^{(n)}(\tau) \right) \E[T(\sigma) T(\tau)] \\
& \qquad \qquad  \qquad + \, \sum_{(\sigma, \tau) \in(S'_n)^c} \varphi\left(\sqrt{n} \bm^{(n)}(\sigma) \right) \varphi\left(\sqrt{n} \bm^{(n)}(\tau) \right) \E[T(\sigma)] \E[T(\tau)],
\end{split}
\end{equation}
and we analyze the two sums in the rhs separately. In order to characterize the configurations with fixed magnetization, we define the set
\begin{equation*}V(\bm',\bm'',\bm''')= \{(\sigma,\tau) \in \mathcal{X}_n^2\, : ||\bm^{(n)}(\sigma)||=\bm', \, ||\bm^{(n)}(\tau)||=\bm'',\, ||\bm^{(n)}(\sigma\tau)||=\bm'''\}.
\end{equation*}
By the three-dimensional local central limit theorem, there exists a universal constant $C>0$
such that
\begin{align*}
    |V(\bm',\bm'',\bm'')| \leq C2^{2n}n^{-3}e^{-\frac{n}{4}(||\bm'||^2+||\bm''||^2+||\bm'''||^2)}.
\end{align*}
The first sum in \eqref{eq:atypicalpairterms} then becomes
\begin{equation} \label{eq:atypicalpairfirst}
\begin{split}
&\sum_{\substack{(\bm,\bm',\bm'') : \\ || \bm'''||>r_n}} \,\,\sum_{(\sigma, \tau) \in V(\bm',\bm'',\bm''')} \varphi\left(\sqrt{n} \bm^{(n)}(\sigma) \right) \varphi\left(\sqrt{n} \bm^{(n)}(\tau) \right)\E[T(\sigma) T(\tau)] \\
&\leq C'(\beta,\varphi)2^{2n} \sum_{\substack{(\bm,\bm',\bm''): 
\\ || \bm'''||>r_n}} \!\!\!\! n^{-3}  e^{-\frac{n}{4}(||\bm'||^2+||\bm''||^2+||\bm'''||^2)} \\
& \qquad \qquad \qquad \qquad \qquad \qquad\cdot e^{
 \left(\frac{\beta}{2}n+O\left(\frac{1}{np^2}\right)\right)
\Big( E_{\alpha_n} (\bm') +E_{\alpha_n} (\bm'') \Big)
 + O\left(\frac{1}{p}\right) E_{\alpha_n} (\bm''')}  \\
 &= C'(\beta,\varphi)2^{2n} \sum_{\substack{(\bm,\bm',\bm''): 
 \\ || \bm'''||>r_n}} \!\!\!\! n^{-3}  e^{-\frac{n}{4}(||\bm'||^2+||\bm''||^2)} \\
 & \qquad \qquad \qquad \qquad \qquad \qquad\cdot e^{
 \left(\frac{\beta}{2}n+O\left(\frac{1}{np^2}\right)\right)
\Big( E_{\alpha_n} (\bm') +E_{\alpha_n} (\bm'') \Big) }e^{-\frac{n}{4}||\bm'''||^2+O\left(\frac{1}{p}\right) E_{\alpha_n} (\bm''')} \\
& = C'(\beta,\varphi)2^{2n}\sum_{\substack{(\bm,\bm',\bm''): 
\\ || \bm''||>r_n}} \!\!\!\! n^{-3}  e^{\frac{n}{2} \left( \beta E_{\alpha_n}(\bm') -\frac{\| \bm'\|^2}{2}+O\left(\frac{\|\bm'\|^2}{(np)^2}\right)\right)} e^{\frac{n}{2} \left( \beta E_{\alpha_n}(\bm'') -\frac{\| \bm'' \|^2}{2}+O\left(\frac{\|\bm''\|^2}{(np)^2}\right)\right)} \\
& \qquad \qquad \qquad \qquad \qquad \quad \,\,\,\,\,\,\, \cdot e^{n \left (-\frac{||\bm'''||^2}{4}+O\left(\frac{1}{np}\right) E_{\alpha_n} (\bm''')\right) },
\end{split}
\end{equation}
for some constant $C'(\beta,\varphi) >0.$ Note that the main terms in the first two exponents above correspond to the main terms in the Taylor expansions of $G_{\alpha_n,\beta}(\bm(\sigma))$ and $G_{\alpha_n,\beta}(\bm(\tau))$ around $G_{\alpha_n,\beta}(\bm^0)$. Apply the change of variables ${\bf x} =\sqrt{n}\bm'$, ${\bf y} =\sqrt{n}\bm''$, and ${\bf z} =\sqrt{n} \bm'''$. Let $\tilde \Sigma_n$ be the $6\times 6$ Hessian matrix of the quadratic form appearing in the exponent. Then the sum over atypical pairs can be expressed as the Riemann sum
\begin{equation*}
\sum_{({\bf x}, {\bf y}, {\bf z}) \in D_n}
\frac{1}{n^3}
 e^{-\frac{1}{2} ({\bf x}, {\bf y}, {\bf z})^\top (\tilde \Sigma_n)^{-1} ({\bf x}, {\bf y}, {\bf z})} (1+o(1)),
\end{equation*}
where
$$
D_n = \left\{ {\bf x}, {\bf y}, {\bf z} \in \sqrt{n}\cdot S_n^c: \|{\bf z}\|^2 > (np)^{\delta}  \right\}.
$$
Hence, for $n\gg 1$, since $(np)^{\delta} \to \infty$, the sum can be approximated by the integral
\begin{equation}
\label{eq:intDN}
\int_{({\bf x}, {\bf y}, {\bf z}) \in D_n}
 e^{-\frac{1}{2} ({\bf x}, {\bf y}, {\bf z}) ^\top (\tilde \Sigma)^{-1} ({\bf x}, {\bf y}, {\bf z})} d({\bf x}, {\bf y}, {\bf z}) =o(1),
\end{equation}
where $\tilde \Sigma$ can be written in block form as
\begin{equation*}
\tilde \Sigma = 
\begin{pmatrix}
\Sigma & 0 & 0 \\
0 & \Sigma & 0 \\
0 & 0 & \Sigma'
\end{pmatrix},
\end{equation*}
with $\Sigma$ the $2 \times 2$ covariance matrix defined in \eqref{def:sigma}, and $\Sigma'$ the \(2 \times 2\) block coming from the cross terms such that $(\Sigma')^{-1}=-\frac{1}{4} \text{Id}_2$. Combining the above \eqref{eq:atypicalpairfirst}-\eqref{eq:intDN} and Lemma   \ref{prop:meantilted}, we obtain
\begin{equation}
\label{eq:atyppairsfirstfinal}
\sum_{(\sigma, \tau) \in(S'_n)^c} \varphi\left(\sqrt{n} \bm(\sigma) \right) \varphi\left(\sqrt{n} \bm(\tau) \right) \E[T(\sigma) T(\tau)] = o\left(2^{2n}\right) = o\left( \E[\widetilde{Z}_n(\varphi)]^2\right).
\end{equation}
Next, we analyze the second sum in the rhs of \eqref{eq:atypicalpairterms}, which from Lemma \ref{lemma:T}(i) and \eqref{eq:binomialLLT}, can be written as
\begin{align*}
&\sum_{\substack{(\bm',\bm'',\bm''') \in \mathbb{Z}^6: \\ || \bm'''||>r_n}} \sum_{(\sigma, \tau) \in V_n} \varphi\left(\sqrt{n} \bm^{(n)}(\sigma) \right) \varphi\left(\sqrt{n} \bm^{(n)}(\tau) \right)\E[T(\sigma)] \E [ T(\tau)] \\
 &\leq C'(\beta, \varphi) 2^{2n} \sum_{\substack{(\bm',\bm'',\bm''') \in \mathbb{Z}^6: \\ || \bm'''||>r_n}} n^{-3} e^{\frac{n}{2} \left( \beta E_{\alpha_n}(\bm') -\frac{\| \bm'\|^2}{2}+O\left(\frac{\|\bm'\|^2}{(np)^2}\right)\right)}  \\
 & \qquad \qquad \qquad \qquad\qquad \qquad \qquad  \cdot e^{\frac{n}{2} \left( \beta E_{\alpha_n}(\bm'') -\frac{\| \bm'' \|^2}{2}+O\left(\frac{\|\bm''\|^2}{(np)^2}\right)\right)} \, e^{-\frac{n}{4} ||\bm'''||^2}.
\end{align*}
Arguing as before, we obtain 
\begin{equation}
\label{eq:atyppairssecondfinal}
\sum_{(\sigma, \tau) \in(S'_n)^c} \varphi\left(\sqrt{n} \bm(\sigma) \right) \varphi\left(\sqrt{n} \bm(\tau) \right) \E[T(\sigma)] \E[T(\tau)] = o\left(2^{2n}\right) = o\left( \E[\widetilde{Z}_n(\varphi)]^2\right).
\end{equation}
By putting together \eqref{eq:vartypicalatypical}, \eqref{eq:typicalpairs}, \eqref{eq:atypicalpairterms}, \eqref{eq:atyppairsfirstfinal}, and \eqref{eq:atyppairssecondfinal}, we conclude that
$$
{\rm Var}\left(\widetilde{Z}_n(\varphi)\right) = o\left( \E[\widetilde{Z}_n(\varphi)]^2\right).
$$
As a consequence, 
$$
\mathbb{E}\left[\left(\frac{\widetilde{Z}_n(\varphi)}{\mathbb{E}[\widetilde{Z}_n(\varphi)]} - 1\right)^2\right] 
= \frac{{\rm Var}(\widetilde{Z}_n(\varphi))}{\mathbb{E}[\widetilde{Z}_n(\varphi)]^2} \rightarrow 0,
$$
which shows that $\frac{\widetilde{Z}_n(\varphi)}{\mathbb{E}[\widetilde{Z}_n(\varphi)]}$ converges to 1 in $L^2$.
\end{proof}

\subsection{Proof of Theorem \ref{thm:CLT}}
\begin{proof}[Proof of Theorem \ref{thm:CLT}]
We argued that proving the CLT reduces to analyzing the limit of $Z_{n,\beta}(\varphi)/Z_{n,\beta}$ for all $\varphi \in C_b(\mathbb{R}^2)$ as $n \to \infty$, where $Z_{n,\beta}(\varphi)$ is the generalized partition function defined in \eqref{def:ZN}. By the definition of tilted partition function $\widetilde{Z}_{n,\beta}(\varphi)$ in \eqref{def:ZNtilted}, we have that
\begin{equation*}
\lim_{n \to \infty} \frac{Z_{n,\beta}(\varphi)}{Z_{n,\beta}} = \lim_{n \to \infty} \frac{\widetilde{Z}_{n,\beta}(\varphi)}{\widetilde{Z}_{n,\beta}(1)} =\lim_{n \to \infty} \frac{\widetilde{Z}_{n,\beta}(\varphi)}{\E[\widetilde{Z}_{n,\beta}(\varphi)]} \frac{\E[\widetilde{Z}_{n,\beta}(\varphi)]}{\E[\widetilde{Z}_{n,\beta}(1)]} \frac{\E[\widetilde{Z}_{n,\beta}(1)]}{\widetilde{Z}_{n,\beta}(1)}  =  \E_X\left[ \varphi(X) \right],
\end{equation*}
where the third equality holds in $\mathbb{P}$-probability thanks to Lemmas \ref{prop:meantilted} and \ref{prop:vartilted}, with $X \sim \mathcal{N}(0,\Sigma)$.
In conclusion, we have that, for all $\varphi \in C_b(\mathbb{R}^2)$, 
$$
\mathbb{E}_{\mu_{n,\beta}}[\varphi(\sqrt{n}\,\bm^{(n)})] = \frac{Z_{n,\beta}(\varphi)}{Z_{n,\beta}} 
\;\underset{n\to\infty}{\longrightarrow}\; \E_X\left[ \varphi(X) \right].
$$
This readily implies that the law $\mathcal{L}(\sqrt{n} \bm^{(n)})$, considered as a random element of the space of probability measures $\mathcal{M}(\mathbb{R}^2)$, converges in probability to the deterministic measure $\mathcal{N}(0, \Sigma)$. This concludes the proof of the theorem.
\end{proof}

%%%%%%%%%%%%%%%%%%%%%%%%%%%%%%%%%%%%%%%%%%%%%%%%%%%%%%%%%%%%%%%%%%%%%%%%%%%%%%%%%%%%%%%%%%%%%%%%%%%%%%%%%%

\end{document}